\documentclass[12pt]{amsart}
\pdfoutput=1

\usepackage{mathtools}
\usepackage{amsmath}
\usepackage{amsthm}
\usepackage{amssymb}
\usepackage{amsbsy}
\usepackage{amstext}
\usepackage{amsopn}
\usepackage{enumerate}
\usepackage{xcolor}
\usepackage{graphicx}
\usepackage{microtype}
\usepackage{verbatim}
\usepackage[margin=1in,marginparwidth=0.8in, marginparsep=0.1in]{geometry}
\usepackage[pagebackref, bookmarks=true, bookmarksopen=true, bookmarksdepth=3,bookmarksopenlevel=2, colorlinks=true, linkcolor=blue, citecolor=blue, filecolor=blue, menucolor=blue, urlcolor=blue]{hyperref}
\usepackage{newtxtext} 
\usepackage{stmaryrd}
\usepackage{accents}
\usepackage{bbm}

\newcommand{\newword}[1]{\textbf{\emph{#1}}}

\DeclareMathOperator{\cl}{cl}
\DeclareMathOperator{\Cone}{Cone}

\newcommand{\bfa}{\mathbf{a}}

\newcommand{\bfg}{\mathbf{g}}
\newcommand{\bfd}{\mathbf{d}}

\newcommand{\cA}{\mathcal{A}}

\newcommand{\cF}{\mathcal{F}}
\newcommand{\cD}{\mathcal{D}}

\newcommand{\into}{\hookrightarrow}
\newcommand{\kk}{\Bbbk}

\newcommand{\ol}[1]{\overline{#1}}
\newcommand{\onto}{\twoheadrightarrow}

\newcommand{\rep}{\operatorname{rep}}

\newcommand{\ZZ}{\mathbb{Z}}

\newcommand{\congto}{\xrightarrow{\sim}}
\newcommand{\wh}{\widehat}

\usepackage{xparse}
\NewDocumentCommand{\dbc}{ m O{#1} }{G^{#1}}

\newcommand{\muoc}{s_1cs_1}
\newcommand{\muoci}{s_1c^{-1}s_1}
\newcommand{\munc}{s_ncs_n}
\newcommand{\munci}{s_nc^{-1}s_n}

\newtheorem{theorem}{Theorem}[section]

\newtheorem{corollary}[theorem]{Corollary}

\newtheorem{lemma}[theorem]{Lemma}
\newtheorem{proposition}[theorem]{Proposition}
\theoremstyle{definition}

\newtheorem{remark}[theorem]{Remark}
\numberwithin{equation}{section}
\numberwithin{figure}{section}

\begin{document}
\title{Affine cluster monomials are generalized minors}

\author[Dylan Rupel]{Dylan Rupel}
\address[Dylan Rupel]{ University of Notre Dame, Department of Mathematics, Notre Dame, IN 46556, USA}
\email{drupel@nd.edu}

\author[Stella]{Salvatore Stella}
\address[Salvatore Stella]{University of Haifa, Departments of Mathematics, Haifa, Mount Carmel, 31905, Israel}
\email{stella@math.haifa.ac.il}

\author[Harold Williams]{Harold Williams}
\address[Harold Williams]{University of California \\ Davis CA, USA}
\email{hwilliams@math.ucdavis.edu}

\begin{abstract}
  We study the realization of acyclic cluster algebras as coordinate rings of Coxeter double Bruhat cells in Kac-Moody groups.
  We prove that all cluster monomials with $\bfg$-vector lying in the doubled Cambrian fan are restrictions of principal generalized minors.
  As a corollary, cluster algebras of finite and affine type admit a complete and non-recursive description via (ind-)algebraic group representations, in a way similar in spirit to the Caldero-Chapoton description via quiver representations.
  In type $A_1^{(1)}\!\!$, we further show that elements of several canonical bases (generic, triangular, and theta) which complete the partial basis of cluster monomials are composed entirely of restrictions of minors.
  The discrepancy among these bases is accounted for by continuous parameters appearing in the classification of irreducible level-zero representations of affine Lie groups.
  We discuss how our results illuminate certain parallels between the classification of representations of finite-dimensional algebras and of integrable weight representations of Kac-Moody algebras.
\end{abstract}

\setcounter{tocdepth}{1}

\maketitle

\tableofcontents

\section{Introduction}
\thispagestyle{empty}
Let $\wh{G}$ be a symmetrizable Kac-Moody group and $c$ a Coxeter element of its Weyl group.
When~$\wh{G}$ is of symmetric type, Coxeter elements are in correspondence with orientations of its Dynkin diagram.
More generally, they are in correspondence with skew-symmetrizable matrices which coincide, up to signs, with the Cartan matrix of~$\wh{G}$ away from the diagonal.
Thus we may associate to~$c$ a cluster algebra, a recursively defined commutative ring equipped with a canonical partial basis whose elements are called cluster monomials~\cite{FZ02}.
The cluster algebras that appear in this way are said to be acyclic.

Adjoining suitable frozen variables, which we call doubled principal coefficients, this acyclic cluster algebra $\cA_{dp}(c)$ can be realized concretely in terms of the group $\wh{G}$: it is the coordinate ring of the Coxeter double Bruhat cell
\[
  G^{c,c^{-1}} := B_+ \dot{c} B_+ \cap B_- \dot{c}^{-1} B_-
\]
in the derived subgroup~$G\subset\wh{G}$ \cite{BFZ05, Wil13a}.
These varieties generalize the space of tridiagonal matrices with unit determinant and non-zero sub- and superdiagonal entries, which is recovered in the $SL_n$ case when~$c$ is the standard Coxeter element.

The cluster monomials of any cluster algebra are labeled by $\bfg$-vectors, a subset of elements of an integer lattice.
In the acyclic setting, this lattice can be naturally identified with the weight lattice of the relevant group: we identify the $\bfg$-vectors of the initial cluster variables with the fundamental weights.
When we realize $\cA_{dp}(c)$ as $\kk[G^{c,c^{-1}}]$, this labeling of the initial cluster variables acquires a Lie-theoretic meaning: they are the restrictions to $G^{c,c^{-1}}$ of the principal generalized minors of fundamental weights.

Recall that a principal generalized minor is a function on $G$ or $\wh{G}$ of the following form (we generalize slightly from \cite{FZ99}).
Fix a weight representation $V$ with an extremal weight $\lambda$, and choose an extremal vector~$v_\lambda\in V_\lambda$ together with a projection $\pi_\lambda: V \onto  \kk v_\lambda$ factoring through the weight projection onto $V_\lambda$.
Then the minor~$\Delta_{V,\lambda}$ is the function whose value at~$g$ is the ratio~$\pi_\lambda(gv_\lambda)/v_\lambda$.
When $\lambda$ is in the Tits cone, for example if it is fundamental, there is a canonical choice of $V$ given by the irreducible highest-weight representation whose highest weight is conjugate to~$\lambda$ under the Weyl group~$W$.
Since here, and in most cases of interest, $\Delta_{V,\lambda}$ is independent of $v_\lambda$ and $\pi_\lambda$, we suppress them from the notation.

As indicated above, the initial cluster variables on $G^{c,c^{-1}}$ coincide by definition with the restrictions of the fundamental principal minors.
More generally, the coordinate ring of any double Bruhat cell possesses a cluster structure in which finitely many cluster variables will by construction coincide with (possibly non-principal) minors \cite{BFZ05, Wil13a}.
The purpose of this paper is to show that the role of generalized minors in cluster theory is in fact much deeper.
Indeed, in finite and affine types we show that they provide a \emph{complete, non-recursive} description of cluster monomials in terms of the representation theory of (ind-)algebraic groups.

\begin{theorem}
  \label{thm:mainthm1}
  Let~$\wh{G}$ be of finite or affine type, and let $x_{\lambda;c} \in \cA_{dp}(c)$ be the cluster monomial of $\bfg\text{-vector}$~$\lambda$.
  Let $V$ be any weight representation of $G$ for which $\lambda$ is extremal, and let~$\Delta_{V,\lambda}$ be the principal generalized minor defined by some choice of $v_\lambda$, $\pi_\lambda$.
  Then the isomorphism~${\cA_{dp}(c) \cong \kk[G^{c,c^{-1}}]}$ identifies the cluster monomial~$x_{\lambda;c}$ with the restriction of the minor~$\Delta_{V,\lambda}$.
\end{theorem}

In the special case where~$x_{\lambda;c}$ is a cluster variable, some instances of the above result are known: when~$\wh{G}$ is a semisimple algebraic group it was shown in~\cite{YZ08}, and when~$\wh{G}$ is of type~$A_n^{(1)}$ or of finitely many other affine types it was shown in~\cite{RSW16}.
It was stated as a conjecture in~\cite{RSW16} that the claim holds for all cluster variables in $\cA_{dp}(c)$.

We note that the statement for cluster monomials does not follow trivially from the statement for cluster variables: deducing the former given the latter amounts to showing that the relevant minors satisfy certain relations when restricted to~$G^{c,c^{-1}}$\hspace{-4pt}, but these relations generally do not hold globally on~$\wh{G}$.
For example, consider~$\wh{G}=SL_3$ with $c=s_1s_2$.
The two cluster variables~$x_{\omega_1;c}$ and~$x_{-\omega_2;c}$ form a cluster in~$\cA_{dp}(c)$ and are the restrictions of the minors $\Delta_{\kk^3,\,\omega_1}$ and $\Delta_{\kk^3,-\omega_2}$, respectively.
Theorem~\ref{thm:mainthm1} asserts that the cluster monomial~$x_{\omega_1-\omega_2;c}=x_{\omega_1;c}x_{-\omega_2;c}$ is the restriction of the minor~$\Delta_{{\tiny \bigwedge\nolimits^{\!2}}\kk^3,\,\omega_1-\omega_2}$.
For a generic matrix $g=(g_{ij}) \in SL_3$, the evaluations $\Delta_{{\tiny \bigwedge\nolimits^{\!2}}\kk^3,\,\omega_1-\omega_2}(g) = g_{11}g_{33}-g_{13}g_{31}$ and~$\Delta_{\kk^3,\,\omega_1}(g)\Delta_{\kk^3,-\omega_2}(g) = g_{11} g_{33}$ are not equal, but they do coincide when~$g$ is tridiagonal.

Theorem~\ref{thm:mainthm1} is a special case of the following more general result.
\begin{theorem}
  \label{thm:mainthm2}
  Let~$\wh{G}$ be any symmetrizable Kac-Moody group, and let $x_{\lambda;c} \in \cA_{dp}(c)$ be a cluster monomial whose $\bfg\text{-vector}$~$\lambda$ is contained in the doubled $c$-Cambrian fan.
  Let $V$ be any weight representation of $G$ for which $\lambda$ is extremal, and let~$\Delta_{V,\lambda}$ be the principal generalized minor defined by some choice of $v_\lambda$, $\pi_\lambda$.
  Then the isomorphism~${\cA_{dp}(c) \cong \kk[G^{c,c^{-1}}]}$ identifies the cluster monomial~$x_{\lambda;c}$ with the restriction of the minor~$\Delta_{V,\lambda}$.
\end{theorem}

The doubled $c$-Cambrian fan~\cite{RS15} is a (generally proper) subset of the weight lattice consisting of cones obtained by gluing regions in the Coxeter arrangement of~$W$ in accordance with certain lattice congruences given by~$c$.
This fan always covers the Tits cone and its opposite, but it is larger than the union of these except in finite type where the Tits cone already contains the entire weight lattice.
Moreover, when~$\wh{G}$ is of finite or affine type, the doubled $c$-Cambrian fan contains the~$\bfg$-vectors of all cluster monomials in~$\cA_{dp}(c)$.

When $\lambda$ is not conjugate to a dominant or antidominant weight, the classification of weight representations for which $\lambda$ is extremal is not well-understood in general.
In affine type, these are the representations of level zero.
The irreducible level-zero representations of~$\widehat{G}$ with finite-dimensional weight spaces were classified in \cite{Cha86, CP86, CP88}.
Unlike the highest-weight case, there is not a unique irreducible level-zero representation with a fixed extremal weight, but rather a continuous family.
In general type, universal extremal weight representations were constructed and studied by Kashiwara \cite{Kas94} in the quantum setting -- in affine type, the irreducible representations from \cite{Cha86, CP86, CP88} are quotients of these.
As a corollary of Theorem \ref{thm:mainthm2}, we see that when $\lambda$ lies in the doubled $c$-Cambrian fan, the minor $\Delta_{V,\lambda}$ is independent of the choice of extremal weight representation~$V$ and of the projection $\pi_\lambda$.

\subsection*{Canonical bases}
Outside of finite type, a cluster algebra is not spanned by its cluster monomials.
A major thread in the development of cluster theory has been the construction of canonical bases which extend the set of cluster monomials.
Examples include the generic basis of \cite{Dup12}, the triangular basis of \cite{BZ14}, and the theta basis of \cite{GHKK14}.
The first two of these bases are closely related, respectively, to the dual semicanonical and dual canonical bases of the coordinate ring of the positive unipotent subgroup $N_+ \subset G$.
In rank two, the third coincides with the greedy basis of \cite{SZ04, LLZ14} by \cite{CGMMRSW15}.

While all of these bases contain the set of cluster monomials, they generally do not coincide outside this subset.
However, the labeling of cluster monomials by $\bfg$-vectors extends to a labeling of any of these bases by the weight lattice.
\begin{theorem}
\label{thm:mainthm3}
  Let~$\wh{G} \cong \wh{LSL}_2$ be the affine Kac-Moody group of type $A_{1}^{(1)}$.
  Then the isomorphism~$\cA_{dp}(c) \cong \kk[G^{c,c^{-1}}]$ identifies all elements of the generic, triangular, and theta bases with restrictions of principal generalized minors $\Delta_{V,\lambda}$ for which $V$ is an irreducible weight representation of $\wh{LSL}_2$ with finite-dimensional weight spaces.
  That is, for each basis and for any weight $\lambda$, we can choose such a representation $V$ so that the restriction of $\Delta_{V,\lambda}$ coincides with the basis element labeled by $\lambda$.
\end{theorem}

Implicit in the statement is that weights which are not $\bfg$-vectors of cluster monomials and to which known canonical basis constructions associate different basis elements are precisely the weights for which the restriction of the principal generalized minor $\Delta_{V,\lambda}$ depends nontrivially on the choice of $V$.
Necessarily, such weights lie in the complement of the Tits cone and its opposite, and as recalled above representations with such extremal weights generally depend on continuous parameters.
By tuning these parameters one obtains the three bases above as well as a family of other bases interpolating between them.

\subsection*{Perspective from quiver representations}
From a broader point of view, part of the conceptual significance of our results is that they provide a deeper insight into certain parallels between the representation theories of Kac-Moody groups and of finite-dimensional algebras.
To explain this, we take~$\wh{G}$ to be symmetric, although a similar discussion can be had in the symmetrizable case by considering species \cite{DR76,Gab73}.

Recall again that in the symmetric case the choice of a Coxeter element~$c\in W$ is equivalent to the choice of an acyclic quiver~$Q$ whose underlying graph is the Dynkin diagram of~$\wh{G}$.
Orthogonal to our discussion so far, the cluster monomials in $\cA_{dp}(c)$ have a complete and non-recursive description in terms of the representation theory of~$Q$ \cite{CC06, CK06}.
Namely, there is a bijection between cluster monomials and rigid representations of $Q$.
This intertwines the labeling of the former by~$\bfg$-vectors with the labeling of the latter by their minimal injective copresentations.
The Laurent expansion of each cluster monomial is equal to a generating function of submodules, sorted by dimension vector, of its associated representation.
In this sense, the representation theory of $Q$ entirely controls the structure of~$\cA_{dp}(c)$.

Recall that representations of~$Q$ are classified by their behavior under the Auslander-Reiten translation~$\tau: \rep Q \to \rep Q$~\cite{ASS06}.
An indecomposable representation~$M$ is said to be preprojective if~$\tau^k(M)$ is projective for some~$k\ge 0$, postinjective if it is of the form~$\tau^k(I)$ for some injective~$I$ and some~$k\ge0$, and regular otherwise.
The quiver $Q$ is of finite type if and only if all indecomposable representations are both preprojective and postinjective.
In general these two classes are distinct, but they are dual in a suitable sense and their structure can be completely understood in a uniform way.
On the other hand, the tame-wild dichotomy implies that regular representations are essentially unclassifiable unless~$Q$ is of affine type, in which case they can be explicitly described.

Consider the parallel nature of this classification with that of the irreducible weight representations of $\wh{G}$ with finite dimensional weight spaces.
The group $\wh{G}$ is of finite type if and only if all irreducible representations are both highest-weight and lowest-weight.
In general these two classes are distinct, but they are dual and their structure can be completely understood in a uniform way.
On the other hand, the classification of irreducible representations which are neither highest- nor lowest-weight is completely open unless $\wh{G}$ is of affine type, in which case they can be explicitly described following \cite{Cha86, CP86, CP88}.

Theorem \ref{thm:mainthm1} explicitly links these two trichotomies: given the~$\bfg$-vector $\lambda$ of a cluster variable in $\cA_{dp}(c)$, a weight representation~$V$ for which~$\lambda$ is extremal is highest-weight, lowest-weight, or neither exactly when the~representation of $Q$ associated to~$x_{\lambda;c}$ is respectively preprojective, postinjective, or regular \cite{RS17, RS15}.
Thus our results demonstrate that the parallel between these classifications is not a superficial one, but reflects an equivalence between certain structural information about representations of~$Q$ and structural information about representations of~$G$.
Or more concretely, the submodule structure of the rigid representations of~$Q$ is explicitly controlled by how the subvariety~$G^{c,c^{-1}} \subset G$ probes the extremal weight representations of~$G$.
If Theorem \ref{thm:mainthm3} provides correct intuition for the general picture, this phenomenon is perhaps not even specific to rigid representations, as the interpretation of cluster monomials in terms of quivers may be extended to these bases by considering potentially non-rigid representations (see Remark~\ref{rem:quiver bases}).

\subsection*{Organization}
In Section \ref{sec:cdbc}, we introduce and study certain regular maps among Coxeter double Bruhat cells of different Coxeter elements.
These maps allow us to geometrize the mutations at or the freezing of a sink or a source.
We show that cluster variables and generalized minors behave well and in a compatible way under these maps.
In Section \ref{sec:bigcambrian}, we use these results to prove Theorem~\ref{thm:mainthm2}.
The argument is a double induction on certain subwords of $c^{\infty}$ and on the rank of the group.
In Section \ref{sec:bases}, we study the principal level-zero minors of irreducible weight representations of $\wh{LSL}_2$ with finite-dimensional weight spaces.
We fully compute their dependence on the continuous parameters appearing in the classification of such representations, and prove they may be specialized to recover various canonical bases.

\textsc{Acknowledgments}
The second author is grateful to Nathan Reading for repeatedly correcting his wrong intuitions on doubled Cambrian fans.
D.R. was partially supported by an AMS-Simons Travel Grant; S.S. was supported by ISF grant 1144/16;  H.W. was supported by NSF Postdoctoral Fellowship DMS-1502845 and NSF grant DMS-1702489.

\section{Coxeter double Bruhat cells}
\label{sec:cdbc}

In this section we study maps between different Coxeter double Bruhat cells of a fixed group, or of a group and subgroup of corank one.
The main results are that both cluster monomials and generalized minors behave well and in a similar way under these maps.

\subsection{Recollections on double Bruhat cells}

We summarize our notation and conventions here, mostly following~\cite{RSW16}.
Let $\kk$ be an algebraically closed field of characteristic zero and~$A=(a_{ij})$ an~$n\times n$ symmetrizable Cartan matrix with Weyl group~$W=\langle s_i\rangle$.
We fix a Coxeter element~$c\in W$ and, in order to simplify our notation, we assume that~$c=s_1\cdots s_n$.
Up to simultaneously permuting the rows and columns of~$A$, there is no loss of generality in doing so and we keep this assumption throughout the paper.

We write $\wh{G}$ for the minimal version of the associated Kac-Moody group as considered in \cite{KP83} and \cite[Section 7.4]{Kum02}.
This is an ind-algebraic group of ind-finite type whose derived subgroup~$G$ is generated by the simple coroot subgroups $\varphi_i: SL_2 \into \wh{G}$ for $i \in [1,n]$.
When $A$ is of untwisted affine type, $G$ can be identified with the universal central extension $\widetilde{LG^{\circ}}$ of the algebraic loop group $LG^\circ$ of the associated semisimple algebraic group $G^{\circ}$.
The full group $\wh{G}$ is then the semidirect product $\wh{LG^{\circ}} := \widetilde{LG^{\circ}} \rtimes \kk^\times$ with the group of loop rotations.

Let $B_+$ and $B_-$ denote the standard opposite Borel subgroups of $G$ and $H$ its maximal torus.
We denote their counterparts in $\wh{G}$ as $\wh{B}_+$, $\wh{B}_-$, and $\wh{H}$ respectively.
We have a decomposition of $G$ into double Bruhat cells
\[
  G^{u,v} = B_+ \dot{u} B_+ \cap B_- \dot{v} B_-,
\]
where $u, v \in W$ and $\dot{u}, \dot{v}$ are arbitrary lifts to $G$.
Since we most often consider the case $u = v^{-1}$, we also introduce the notation
\[
  \dbc{w} := G^{w,w^{-1}}.
\]

We use the following notation to discuss specific elements of $G$.
For $t \in \kk^\times$ and $i \in[1,n]$, we set
\begin{gather*}
  x_{i}(t):=\varphi_i\begin{pmatrix} 1 & t \\ 0 & 1\end{pmatrix}
  \quad
  t^{\alpha_i^\vee}:=\varphi_i\begin{pmatrix}t & 0 \\ 0 & t^{-1}\end{pmatrix}
  \quad
  x_{\ol{\imath}}(t):=\varphi_i\begin{pmatrix} 1 & 0 \\ t & 1\end{pmatrix}
  \quad
  \ol{s_{i}} := \varphi_i \begin{pmatrix} 0 & -1 \\ 1 & 0 \end{pmatrix}.
\end{gather*}
For $w \in W$, we set $\ol{w} := \ol{s_{i_1}}\cdots\ol{s_{i_k}}$, where $s_{i_1}\cdots s_{i_k}$ is any reduced expression for $w$.

The weight lattice $P$ of $G$ is spanned by the fundamental weights $\{\omega_i\}_{i \in [1,n]}$.
It embeds into the weight lattice $\wh{P}$ of $\wh{G}$, but the two are not isomorphic except in finite type when $H = \wh{H}$.
In untwisted affine type, $P$ is identified with the direct sum of the weight lattice $P^\circ$ of $G^\circ$ and the group $\kk \kappa$ of characters of the $\kk^\times$ appearing in the central extension, while $\wh{P}$ is the further direct sum with the group $\kk \delta$ of characters of the $\kk^\times$ of loop rotations.

If $V$ is a weight representation of $G$, a weight vector $v_\lambda \in V$ is said to be \newword{extremal} if $\dot{w}v_\lambda$ is either a highest- or lowest-weight vector of $\varphi_i(SL_2)$ for every $i\in[1,n]$ and every $w \in W$ \cite{Kas94}.
A weight $\lambda$ is extremal if each $v_\lambda$ in the weight space $ V_\lambda$ is an extremal vector.
Suppose $\lambda$ is an extremal weight of $V$, and choose an extremal vector~$v_\lambda\in V_\lambda$ together with a projection $\pi_\lambda: V \onto  \kk v_\lambda$ factoring through the weight projection onto $V_\lambda$.
Then the principal generalized minor~$\Delta_{V,\lambda}$ is the function whose value at an element~$g\in G$ is the ratio~$\pi_\lambda(gv_\lambda)/v_\lambda$.
We omit $v_\lambda$ and $\pi_\lambda$ from the notation for efficiency, since for our purposes they will generally be irrelevant.

When $\lambda$ is dominant (resp.\ antidominant), we let $V(\lambda)$ denote the irreducible representation with highest (resp.\ lowest) weight $\lambda$.
In this case, we abbreviate $\Delta_{V(\lambda),\lambda}$ to $\Delta_\lambda$.
Given $u, v \in W$, we also have the (non-principal, if $u \neq v$) generalized minor
\[
  \Delta^{u\omega_i}_{ v\omega_i}(g) := \Delta_{\omega_i}(\ol{u}^{\raisebox{1pt}{$\scriptstyle -1$}}g\ol{v}).
\]

Let $B_c=(b_{ij})$ denote the $n\times n$ skew-symmetrizable matrix given by
\begin{equation}
  \label{eq:coxeter exchange matrix}
  b_{ij}=\begin{cases}
    0 & \text{if $i=j$;}\\
    a_{ij} & \text{if $j < i$;}\\
    -a_{ij} & \text{if $i < j$.}
  \end{cases}
\end{equation}
We write $\cA(c)$ for the coefficient-free cluster algebra over $\kk$ associated to $B_c$.
Recall that~it~is~generated by elements called \newword{cluster variables} grouped into overlapping $n$-element subsets called \newword{clusters} \cite{FZ02}.
The cluster algebra $\cA_{dp}(c)$ with \newword{doubled principal coefficients} is the cluster algebra of geometric type with initial exchange matrix
\[
  \left[
    \begin{array}{c}
      B_c\\
      Id_n\\
      Id_n
    \end{array}
  \right].
\]
We denote by $x_{\omega_i;c}$ its $i$-th initial cluster variable and by $z_{i;c}$ (resp.\ $z_{\ol{\imath};c}$) its $i$-th (resp.\ $(n+i)$-th) frozen variable.
In this setting, frozen variables are invertible.

For any $u, v \in W$, there is an upper cluster structure on the coordinate ring of $G^{u,v}$ \cite{BFZ05,Wil13a}.
In the case of $\dbc{c}$, a monomial transformation of frozen variables identifies this cluster structure as having doubled principal coefficients.
\begin{proposition}
  \cite{YZ08,RSW16}
  \label{prop:coordring}
  There is an isomorphism $\cA_{dp}(c) \cong \kk[\dbc{c}]$ that identifies the initial cluster variables and the frozen variables with the following restrictions of monomials in generalized minors:
  \begin{equation*}
    x_{\omega_i;c} = \Delta_{\omega_i}\Big|_{\dbc{c}},
    \qquad
     z_{i;c} = \Delta^{\omega_i}_{ c\omega_i} \prod_{j < i}(\Delta^{\omega_j}_{c \omega_j})^{a_{ji}}\Big|_{\dbc{c}},
    \qquad
     z_{\ol{\imath};c} = \Delta^{c \omega_i}_{\omega_i} \prod_{j < i}(\Delta^{c \omega_j}_{\omega_j})^{a_{ji}}\Big|_{\dbc{c}}.
  \end{equation*}
\end{proposition}

The algebra $\cA_{dp}(c)$ is $\ZZ^n$-graded and its cluster variables are homogeneous elements \cite{FZ07}.
In view of \cite[Conjecture 2.3]{NS14} and \cite[Theorem 1.2]{NZ12}, we identify $\ZZ^n$ with the weight lattice~$P$ by sending unit vectors to fundamental weights.
With this convention, the degree of the initial cluster variable $x_{\omega_i;c}$ is $\omega_i$ while the degrees of $z_{j;c}$ and $z_{\ol{\jmath};c}$ are chosen in any way so that the degree of the product $z_{j;c}z_{\ol{\jmath};c}$ equals $-\sum_{i=1}^n b_{ij}\omega_i$.

The degree of a cluster variable and more generally of a \newword{cluster monomial} -- that is a product of cluster variables from a given cluster -- is called its \newword{$\bfg$-vector}.
These parametrize cluster monomials \cite{DWZ10,Dem10} and we denote by $x_{\lambda;c}$ the cluster monomial of $\cA_{dp}(c)$ whose $\bfg$-vector is $\lambda \in P$.

The set of $\bfg$-vectors can be organized to give a combinatorial shadow of $\cA_{dp}(c)$: its \newword{$\bfg$-vector fan}.
This is a simplicial fan in $P$ whose rays are spanned by the $\bfg$-vectors of cluster variables and whose maximal cones correspond to clusters.
Our main result applies to cluster monomials whose $\bfg$-vector is contained in a notable subset of the $\bfg$-vector fan, the doubled $c$-Cambrian fan, that we will recall in Section \ref{sec:bigcambrian}.

We conclude this summary with the following result which facilitates the explicit evaluation of functions on $\dbc{c}$.
\begin{proposition}
  \cite{FZ99,Wil13b}
  \label{prop:shuffled factorizations}
  For any shuffling~$(i_1,\dots, i_{2n})$ of the tuples~$(n,\dots,1)$ and~$(\ol{1},\dots,\ol{n})$, there is an open embedding
  \[
    \begin{array}{ccc}
      H \times (\kk^\times)^{2n} & \longrightarrow & \dbc{c} \\
      (h,t_1,\dots,t_{2n}) & \longmapsto & h x_{i_1}(t_1) \cdots  x_{i_{2n}}(t_{2n}).
    \end{array}
  \]
  In particular, a generic element of $\dbc{c}$ admits a factorization of this form.
\end{proposition}

Since the 1-parameter subgroups $x_i(t)$ and $x_{\ol{\imath}}(t)$ are preserved by the adjoint action of $H$, one can modify this generic factorization in various ways.
That is, up to automorphisms of $H \times (\kk^\times)^{2n}$ we can place the term $h$ wherever we like on the right hand side or, more generally, write it as a product and distribute its factors independently.

\subsection{Sink and source mutation}
Given a Coxeter element $c=s_1\cdots s_n$, we consider the conjugate Coxeter elements $\muoc$ and $\munc$.
The $n \times n$ exchange matrices $B_{\muoc}$ and $B_{\munc}$ are obtained from~$B_c$ by sink and source mutations in directions 1 and $n$, respectively.
There are thus isomorphisms~$\cA(\muoc)\congto\cA(c)$ and $\cA(\munc)\congto\cA(c)$ of coefficient-free cluster algebras.

The first goal of this section is to show that these isomorphisms persist under the appension of doubled principal coefficients and to describe the resulting isomorphisms of double Bruhat cells explicitly in Lie-theoretic terms.
We further show that the restrictions of certain minors to $\dbc{c}$ are pulled back under these isomorphisms from restrictions of minors to $\dbc{\muoc}[\muoci]$ and $\dbc{\munc}[\munci]$.
As a consequence, to show that a cluster monomial in $\cA_{dp}(c)$ is the restriction of a minor it is often sufficient to show the corresponding statement for $\cA_{dp}(\muoc)$ or $\cA_{dp}(\munc)$.

Let $\psi_{1;c}:\dbc{c}\dashrightarrow \dbc{\muoc}[\muoci]$ denote the rational map which sends a generic element
\begin{equation*}
  x_{\ol{1}}(t_{\ol{1}}) x_{\ol{2}}(t_{\ol{2}}) \cdots x_{\ol{n}}(t_{\ol{n}}) h x_n(t_n) \cdots x_2(t_2) x_1(t_1)
\end{equation*}
of $\dbc{c}$ to the following generic element of $\dbc{s_1cs_1}[s_1c^{-1}s_1]$:
\begin{equation*}
  x_1(t_{\ol{1}}) t_{\ol{1}}^{\alpha_1^\vee} x_{\ol{2}}(t_{\ol{2}}) \cdots x_{\ol{n}}(t_{\ol{n}}) h x_n(t_n) \cdots x_2(t_2) t_1^{\alpha_1^\vee} x_{\ol{1}}(t_1).
\end{equation*}
We have the following result, which in finite type is a more detailed version of \cite[Corollary~5.10 and Remark~5.11]{YZ08}.
\begin{proposition}
  \label{prop:psi}
  The map $\psi_{1;c}$ is a biregular isomorphism.
  The induced isomorphism
  \[
    \psi_{1;c}^*:\cA_{dp}(\muoc)\congto\cA_{dp}(c)
  \]
  sends cluster monomials to cluster monomials, up to multiplication by frozen variables.
  Specifically,
  \begin{enumerate}
    \item $\psi_{1;c}^*$ sends any cluster variable $x_{\lambda;\muoc}$ with $\lambda \neq -\omega_1$ to the cluster variable $x_{s_1\lambda;c}$;
    \item $\psi_{1;c}^*$ sends the cluster variable $x_{-\omega_1;\muoc}$ to the product $x_{\omega_1;c}z_{1;c}^{-1}z_{\ol{1};c}^{-1}$;
    \item $\psi_{1;c}^*$ sends the frozen variable $z_{i;\muoc}$ (resp.\ $z_{\ol{\imath};\muoc}$) to $z_{i;c}z_{1;c}^{-a_{1i}}$ (resp.\ $z_{\ol{\imath};c}z_{\ol{1};c}^{-a_{1i}}$).
  \end{enumerate}
\end{proposition}

Before proving Proposition \ref{prop:psi} we establish the following useful property of the $\bfg$-vectors of cluster variables in $\cA_{db}(c)$.
\begin{lemma}
  \label{lem:canopy}
  Let $\lambda = \sum_{i=1}^n \lambda_i \omega_i \in P$ be the $\bfg$-vector of a cluster monomial $x_{\lambda;c}$ in $\cA_{dp}(c)$.
  Then~$\lambda_1>0$ if and only if $\lambda=\omega_1$ and $\lambda_n<0$ if and only if $\lambda=-\omega_n$.
\end{lemma}
\begin{proof}
  Let $\beta$ be the denominator vector of $x_{\lambda;c}$.
  By \cite[Proposition 5]{RuSt17}, if $x_{\lambda;c}$ is an initial cluster variable then $\beta$ is a negative simple root, otherwise $\beta$ is a positive root.
  Write $[\beta:\alpha_j]$ for the coefficient of $\alpha_j$ in the expansion of $\beta$ in the basis of simple roots.
  Our choice of Coxeter element, together with \cite[Proposition 9]{RuSt17}, immediately implies that $\lambda_1$ equals $-[\beta:\alpha_1]$ and the first claim follows.

  The same results also imply that $\lambda_n$ is $-[\beta:\alpha_n]$ when $x_{\lambda;c}$ is an initial cluster variable, and is $[s_n(\beta):\alpha_n]$ otherwise.
  In particular, the only case in which $\lambda_n$ is negative is when $\beta = \alpha_n$.
  This happens precisely when $\lambda = -\omega_n$, again by \cite[Proposition 9]{RuSt17}, and our second claim is established.
\end{proof}

\begin{proof}[Proof of Proposition \ref{prop:psi}]
  The cluster variables and frozen variables are regular functions on $\dbc{c}$ and~$\dbc{s_1 cs_1}$ which generate their respective coordinate rings.
  Thus the biregularity of $\psi_{1;c}$ will follow once we confirm that it pulls back these functions as in the claim.

  We begin by studying the action of $\psi_{1;c}^*$ on frozen variables.
  To compute this, we observe that~$s_1cs_1\omega_i=\omega_i-\beta_i$ for the roots
  \[
    \beta_i
    :=
    \begin{cases}
      s_2s_3\cdots s_{i-1}\alpha_i & \text{if $i\ne1$;}\\
      s_2s_3\cdots s_n\alpha_1 & \text{if $i=1$.}
    \end{cases}
  \]

  Let $g\in \dbc{c}$ and $\tilde{g} \in \dbc{\muoc}[\muoci]$ denote generic elements factored as above.
  Considering the action of coroot subgroups, it is easy to see that (cf.~proof of \cite[Lemma~3.4]{YZ08}):
  \begin{equation}
    \label{eq:generic evaluation}
    \Delta^{\omega_i}_{\muoc\omega_i}(\tilde{g})
    =
    \begin{cases} 
      h^{\omega_i}t_1^{\langle \muoc\omega_i|\alpha_1^\vee\rangle}\prod\limits_{2\le j\leq i}t_j^{[\beta_i:\alpha_j]} & \text{if $i\ne1$;}\\
      h^{\omega_1-\alpha_1}t_1^{\langle \muoc\omega_1|\alpha_1^\vee\rangle}\prod\limits_{2\le j\leq n}t_j^{[\beta_1:\alpha_j]} & \text{if $i=1$;}
    \end{cases}
  \end{equation}
  where $\alpha_i^\vee$ denotes the $i$-th simple coroot and $\langle\cdot|\cdot\rangle$ is the pairing satisfying $\langle \omega_i|\alpha_j^\vee\rangle=\delta_{ij}$.

  By Proposition~\ref{prop:coordring}, under the isomorphism $\cA_{dp}(\muoc)\cong\kk[\dbc{\muoc}[\muoci]]$ the frozen variables of~$\cA_{dp}(\muoc)$ are given by
  \[
    z_{i;\muoc}
    =
    \begin{cases}
      \Delta^{\omega_i}_{\muoc\omega_i} \prod_{2\le j<i}(\Delta^{\omega_j}_{\muoc\omega_j})^{a_{ji}} & \text{if $i\ne1$;}\\
      \Delta^{\omega_1}_{\muoc\omega_1} \prod_{2\le j\le n}(\Delta^{\omega_j}_{\muoc\omega_j})^{a_{j1}} & \text{if $i=1$.}
    \end{cases}
  \]
  Applying \eqref{eq:generic evaluation}, we obtain
  \[
    z_{i;\muoc}(\tilde g)
    =
    \begin{cases}
      h^{\omega_i}t_1^{\langle \muoc\omega_i|\alpha_1^\vee\rangle} t_i \prod\limits_{2\le j<i}\left(t_j^{[\beta_i:\alpha_j]}h^{a_{ji}\omega_j}t_1^{a_{ji}\langle \muoc\omega_j|\alpha_1^\vee\rangle}\prod\limits_{2\le k\le j}t_k^{a_{ji}[\beta_j:\alpha_k]}\right) & \text{if $i\ne 1$;}\\
      h^{\omega_1-\alpha_1}t_1^{\langle \muoc\omega_1|\alpha_1^\vee\rangle}\prod\limits_{2\le j\le n}\left(t_j^{[\beta_1:\alpha_j]}h^{a_{j1}\omega_j}t_1^{a_{j1}\langle \muoc\omega_j|\alpha_1^\vee\rangle}\prod\limits_{2\le k\le j}t_k^{a_{j1}[\beta_j:\alpha_k]}\right) & \text{if $i=1$.}
    \end{cases}
  \]
  But observe that when $i\ne1$ the exponent of $t_1$ is given by
  \begin{align*}
    \langle \muoc\omega_i+\sum\limits_{2\le j\le i}a_{ji}\muoc\omega_j|\alpha_1^\vee\rangle
    &=\langle s_2s_3\cdots s_i\omega_i+\sum\limits_{2\le j\le i}a_{ji}s_2s_3\cdots s_j\omega_j|\alpha_1^\vee\rangle\\
    &=\langle -s_2s_3\cdots s_{i-1}\alpha_i-\sum\limits_{2\le j\le i}a_{ji}s_2s_3\cdots s_{j-1}\alpha_j|\alpha_1^\vee\rangle\\
    &=\langle -s_2s_3\cdots s_{i-1}\alpha_i+\sum\limits_{2\le j\le i}s_2s_3\cdots s_{j-1}(s_j\alpha_i-\alpha_i)|\alpha_1^\vee\rangle\\
    &=\langle -\alpha_i|\alpha_1^\vee\rangle=-a_{1i}.
  \end{align*}
  Note that essentially the same calculation in the case $i=1$ leads to the exponent $\langle\omega_1-\alpha_1|\alpha_1^\vee\rangle=-1$.

  By a similar calculation, the exponent of $t_k$ for $2\le k\le n$ is given by
  \[
    [\beta_i:\alpha_k]+\sum\limits_{k\le j< i}a_{ji}[\beta_j:\alpha_k]=[s_2s_3\cdots s_{k-1}\alpha_i:\alpha_k]=\delta_{ki}.
  \]
  Combining these observations, we see that
  \begin{align*}
    z_{i;\muoc}(\tilde{g})
    &=
    \begin{cases}
      t_1^{-a_{1i}}t_i h^{\omega_i}\prod\limits_{2\le j<i} h^{a_{ji}\omega_j}& \quad\,\, \text{if $i\ne 1$;}\\
      t_1^{-1}h^{\omega_1-\alpha_1} \prod\limits_{2\le j\le n} h^{a_{j1}\omega_j}& \quad\,\, \text{if $i=1$;}
    \end{cases}\\
    &=
    \begin{cases}
      t_1^{-a_{1i}}h^{-a_{1i}\omega_1} t_i h^{\omega_i}\prod\limits_{1\le j<i} h^{a_{ji}\omega_j} & \text{if $i\ne 1$;}\\
      t_1^{-1}h^{-\omega_1} & \text{if $i=1$;}
    \end{cases}
  \end{align*}
  where the lower equality uses the identity $h^{\alpha_1}=h^{\sum_{j=1}^n a_{j1}\omega_j}$ for $h \in H$ (note the distinction between $H$ and $\wh{H}$).
  Comparing the last expression with~\cite[Equation (3.10)]{RSW16}, it follows that~$z_{i;\muoc}(\tilde{g})=z_{i;c}(g)z_{1;c}(g)^{-a_{1i}}$ for all $i$.
  The claim for $z_{\ol{\imath};\muoc}$ is obtained by the same calculations with each $t_j$ replaced by $t_{\ol{\jmath}}$ throughout.

  Turning towards cluster variables, we immediately observe that $x_{\omega_i;\muoc}(\tilde{g})=x_{\omega_i;c}(g)$ for $i\ne1$ while
  \[
    x_{\omega_1;\muoc}(\tilde{g})
    =
    t_1t_{\ol{1}}h^{\omega_1}+h^{\omega_1-\alpha_1}=x_{\omega_1-\alpha_1;c}(g).
  \]
  In particular, the map $\psi^*_{1;c}:\cA_{dp}(\muoc)\congto\cA_{dp}(c)$ takes initial cluster variables of $\cA_{dp}(\muoc)$ to the cluster variables of $\cA_{dp}(c)$ obtained by mutating its initial seed in direction 1.
  As we already observed, $B_{\muoc}$ is obtained from $B_c$ by mutation in direction $1$.
  Thus the map $\psi^*_{1;c}$ may be realized as a coefficient specialization from the doubled principal coefficient cluster algebra $\cA_{dp}(\muoc)$ to the cluster algebra $\cA_{dp}(c)$ (viewed here as a cluster algebra with non-principal coefficients since we have mutated the initial seed).

  As in \cite[Proposition 4.5]{YZ08}, this implies that each cluster variable of $\cA_{dp}(\muoc)$ transforms via $\psi^*_{1;c}$ into the corresponding cluster variable of $\cA_{dp}(c)$, possibly multiplied by a monomial in the frozen variables.
  Upon mutating back to the standard initial cluster of $\cA_{dp}(c)$, Lemma~\ref{lem:canopy} together with \cite[Equation (4.2)]{NZ12} shows that the $\bfg$-vectors of cluster variables transform as implied by our statement.

  It remains to compute the monomial in frozen variables attached to each cluster variable.
  According to \cite[Proposition 4.5]{YZ08}, these are obtained by tropically evaluating the respective~$F$-polynomials at $y_k=z_{k;c}z_{\ol{k};c}z_{1;c}^{-a_{ki}}z_{\ol{1};c}^{-a_{ki}}$ for $k\in[1,n]$.
  By \cite[Proposition 11]{RuSt17}, the result of these evaluations is always equal to $1$ except for the case $\lambda=-\omega_1$, when we get $z_{1;c}^{-1}z_{\ol{1};c}^{-1}$.
  This completes the proof.
\end{proof}

Let $\theta: G \congto G$ be the involutive automorphism characterized by
\[
  \theta(h) = h^{-1}, \qquad \theta(x_i(t)) = x_{\ol{i}}(t), \qquad \theta(x_{\ol{i}}(t)) = x_{i}(t),
\]
for $h\in H$, $i\in[1,n]$, and $t\in\kk^\times$.
Clearly, the map $\theta$ restricts to an isomorphism $\dbc{c} \congto \dbc{c^{-1}}[c]$.
The following proposition is a straightforward consequence of the definition of $\psi_{1;c}$.
\begin{proposition}
  The inverse of $\psi_{1;c}:\dbc{c}\to \dbc{\muoc}[\muoci]$ is $\theta \circ \psi_{1;\muoci} \circ \theta$.
  Equivalently, $\psi_{1;c}^{-1}$ is characterized by sending the generic element
  \[
    x_1(t_1) x_{n}(t_{n}) \cdots x_2(t_2) h x_{\ol{2}}(t_{\ol{2}}) \cdots x_{\ol{n}}(t_{\ol{n}}) x_{\ol{1}}(t_{\ol{1}})
  \]
  of $\dbc{\muoc}[\muoci]$ to the following generic element of $\dbc{c}$:
  \[
    x_{\ol{1}}(t_1) t_1^{-\alpha_1^\vee} x_{n}(t_{n}) \cdots x_2(t_2) h x_{\ol{2}}(t_{\ol{2}}) \cdots x_{\ol{n}}(t_{\ol{n}}) t_{\ol{1}}^{-\alpha_1^\vee} x_1(t_{\ol{1}}).
  \]
\end{proposition}

We now discuss the behavior of principal minors under the map $\psi_{1;c}$ and its inverse, or rather the inverse of $\psi_{n;\munc}:\dbc{\munc}[\munci]\dashrightarrow \dbc{c}$.

\begin{lemma}
\label{lemma:one extra step}
Let $V$ be a weight representation of $G$ for which $\lambda = \sum_{i=1}^n \lambda_i \omega_i$ is extremal, and let~$\Delta_{V, \lambda}$ be the minor defined by some choice of $v_\lambda$, $\pi_\lambda$.
  If $\lambda_i \geq 0$ then
  \[
    \Delta_{V,\lambda}\big(x_{\ol{\imath}}(t)g\big)
    =
    \Delta_{V,\lambda}(g)
    =
    \Delta_{V,\lambda}\big(gx_i(t)\big)
  \]
  for every $g\in G$ and $t\in\kk^\times$.
  Similarly, if $\lambda_i \leq 0$ then
  \[
    \Delta_{V,\lambda}\big(x_{i}(t)g\big)
    =
    \Delta_{V,\lambda}(g)
    =
    \Delta_{V,\lambda}\big(gx_{\ol{\imath}}(t)\big).
  \]
\end{lemma}
\begin{proof}
  Consider the case $\lambda_i \geq 0$, the other being similar.
  The fact that $v_\lambda$ is extremal then implies that it is a highest-weight vector for $\varphi_i(SL_2)$.
  Thus $x_i(t) v_\lambda=v_\lambda$ and we have
  \[
    \Delta_{V,\lambda}\big(gx_i(t)\big)
    =
    \pi_\lambda\big( gx_i(t) v_\lambda \big)/v_\lambda
    =
    \pi_\lambda\big( gv_\lambda \big)/v_\lambda
    =
    \Delta_{V,\lambda}(g).
  \]
  On the other hand, since $\lambda$ is extremal and $\lambda_i\geq0$, the component of $x_{\ol{\imath}}(t)gv_\lambda$ of weight $\lambda$ is the same as the component of $gv_\lambda$ of weight $\lambda$.
  Since $\pi_\lambda$ is assumed to factor through the weight projection onto $V_\lambda$, it follows that $\Delta_{V,\lambda}\big(x_{\ol{\imath}}(t)g\big)=\Delta_{V,\lambda}(g)$.
\end{proof}

\begin{proposition}
\label{prop:psiminors}
  Let $V$ be a weight representation of $G$ for which $\lambda = \sum_{i=1}^n \lambda_i \omega_i$ is extremal, and let $\Delta_{V, \lambda}$ be the minor defined by some choice of $v_\lambda$, $\pi_\lambda$
  Let $\Delta_{V,s_1\lambda}$ and $\Delta_{V,s_n\lambda}$ be the minors defined by conjugating this choice of $v_\lambda$, $\pi_\lambda$ by any lifts to $G$ of $s_1$ and $s_n$, respectively.
  Then if~$\lambda_1 \leq 0$ we have
  \[
    \psi_{1;c}^* \Big(\Delta_{V,s_1\lambda}\Big|_{\dbc{\muoc}[\muoci]}\Big) = \Delta_{V,\lambda}\Big|_{\dbc{c}},
  \]
  and if $\lambda_n \geq 0$ we have
  \[
    (\psi_{n;s_ncs_n}^{-1})^* \Big(\Delta_{V,s_n\lambda}\Big|_{\dbc{\munc}[\munci]}\Big) = \Delta_{V,\lambda}\Big|_{\dbc{c}}.
  \]
\end{proposition}
\begin{proof}
  We prove the first assertion, the second being similar.
  It suffices to consider the evaluation of both sides on the generic element
  \[
    g = x_{\ol{1}}(t_{\ol{1}}) x_{\ol{2}}(t_{\ol{2}}) \cdots x_{\ol{n}}(t_{\ol{n}}) h x_n(t_n) \cdots x_2(t_2) x_1(t_1).
  \]
  For convenience, we set $g_0 = x_{\ol{2}}(t_{\ol{2}}) \cdots x_{\ol{n}}(t_{\ol{n}}) h x_n(t_n) \cdots x_2(t_2)$ so that
  \[
    g = x_{\ol{1}}(t_{\ol{1}}) g_0 x_1(t_1), \qquad \psi_{1;c}(g) = x_1(t_{\ol{1}}) t_{\ol{1}}^{\alpha_1^\vee} g_0 t_1^{\alpha_1^\vee} x_{\ol{1}}(t_1).
  \]

  Since $\lambda_1 \leq 0$, the coefficient of $\omega_1$ in $s_1\lambda$ is non-negative and we can apply Lemma~\ref{lemma:one extra step} to obtain
  \begin{align*}
    \Delta_{V,s_1\lambda}(\psi_{1;c}(g))
    &=
    \Delta_{V,s_1\lambda}\big(x_1(t_{\ol{1}}) t_{\ol{1}}^{\alpha_1^\vee} g_0 t_1^{\alpha_1^\vee} x_{\ol{1}}(t_1)\big)\\
    &=
    \Delta_{V,s_1\lambda}\big(x_{\ol{1}}(-t_{\ol{1}}^{-1}) x_1(t_{\ol{1}}) t_{\ol{1}}^{\alpha_1^\vee} g_0 t_1^{\alpha_1^\vee} x_{\ol{1}}(t_1) x_1(-t_1^{-1})\big).
  \end{align*}
  On the other hand, by an elementary computation in $SL_2$ we have
  \[
    x_{\ol{1}}(-t_{\ol{1}}^{-1}) x_1(t_{\ol{1}}) t_{\ol{1}}^{\alpha_1^\vee} = \ol{s_1}^{-1} x_{\ol{1}}(t_{\ol{1}}), \qquad t_1^{\alpha_1^\vee} x_{\ol{1}}(t_1) x_1(-t_1^{-1}) = x_1(t_1) \ol{s_1}.
  \]
  It then follows that
  \begin{align*}
    \Delta_{V,s_1\lambda}\big(x_{\ol{1}}(-t_{\ol{1}}^{-1}) x_1(t_{\ol{1}}) t_{\ol{1}}^{\alpha_1^\vee} g_0 t_1^{\alpha_1^\vee} x_{\ol{1}}(t_1) x_1(-t_1^{-1})\big)
    &=
    \Delta_{V,s_1\lambda}\left( \ol{s_1}^{-1} x_{\ol{1}}(t_{\ol{1}}) g_0 x_1(t_1) \ol{s_1} \right)\\
    &=
    \Delta_{V,\lambda}\left(x_{\ol{1}}(t_{\ol{1}}) g_0 x_1(t_1)\right),
  \end{align*}
  since by definition $\Delta_{V, s_1 \lambda}(g') = \Delta_{V, \lambda}(\ol{s_1}^{-1} g' \ol{s_1})$ for any $g' \in G$ (note that $\Delta_{V, \lambda}(\ol{s_1}^{-1} g' \ol{s_1}) = \Delta_{V, \lambda}(\dot{s_1}^{-1} g' \dot{s_1})$ for any other lift $\dot{s}_1$ of $s_1$ to $G$).
\end{proof}

\subsection{Restriction to corank one}
Our next task is to lay the groundwork for proving our main result via induction on the rank of the group $\wh{G}$.
The key point is that if we delete a sink or a source from the exchange matrix $B_c$, the resulting cluster algebra of smaller rank embeds in a Lie-theoretically meaningful way into $\cA_{dp}(c)$.
In particular, we can reduce certain questions about restrictions of minors to $\dbc{c}$ to questions about minors of a proper subgroup.
To this end, we will write $\wh{G}_{\langle k\rangle}$ for the subgroup of $\wh{G}$ whose Dynkin diagram is obtained by deleting vertex $k$ from that of $\wh{G}$ and we adopt similar notations of all the other objects related to $\wh{G}_{\langle k\rangle}$.

Write $\eta_{1;c}: \dbc{c}\dashrightarrow \dbc{s_1c}[c^{-1}s_1]_{\langle 1\rangle}$ for the rational map which sends a generic element
\[
  x_{\ol{1}}(t_{\ol{1}}) x_{\ol{2}}(t_{\ol{2}}) \cdots x_{\ol{n}}(t_{\ol{n}}) h_1^{\alpha_1^\vee} h_2^{\alpha_2^\vee} \cdots h_n^{\alpha_n^\vee} x_n(t_n) \cdots x_2(t_2) x_1(t_1)
\]
of $\dbc{c}$ to the following generic element of $\dbc{s_1c}[c^{-1}s_1]_{\langle 1 \rangle}$:
\[
  x_{\ol{2}}(t_{\ol{2}}) \cdots x_{\ol{n}}(t_{\ol{n}}) h_2^{\alpha_2^\vee} \cdots h_n^{\alpha_n^\vee} x_n(t_n h_1^{a_{1n}}) \cdots x_2(t_2 h_1^{a_{12}}).
\]
Similarly, write $\eta_{n;c}: \dbc{c}\dashrightarrow \dbc{cs_n}[s_nc^{-1}]_{\langle n \rangle}$ for the rational map $\theta_{\langle n \rangle}\circ\eta_{n,c^{-1}}\circ\theta$ sending a generic element
\[
  x_n(t_n) x_{n-1}(t_{n-1}) \cdots x_1(t_1) h_1^{\alpha_1^\vee} \cdots h_{n-1}^{\alpha_{n-1}^\vee} h_n^{\alpha_n^\vee} x_{\ol{1}}(t_{\ol{1}}) \cdots x_{\ol{n-1}}(t_{\ol{n-1}}) x_{\ol{n}}(t_{\ol{n}})
\]
of $\dbc{c}$ to the following generic element of $\dbc{cs_n}[s_nc^{-1}]_{\langle n \rangle}$:
\[
  x_{n-1}(t_{n-1}) \cdots x_1(t_1) h_1^{\alpha_1^\vee} \cdots h_{n-1}^{\alpha_{n-1}^\vee} x_{\ol{1}}(t_{\ol{1}}h_n^{-a_{n1}}) \cdots x_{\ol{n-1}}(t_{\ol{n-1}}h_n^{-a_{n,n-1}}).
\]

\begin{proposition}
  \label{proposition:injection}
  The maps $\eta_{1;c}$ and $\eta_{n;c}$ are regular and dominant.
  The pullbacks
  \[
    \eta_{1;c}^*: \cA_{dp}(s_1c) \hookrightarrow \cA_{dp}(c)
    \qquad
    \text{and}
    \qquad
    \eta_{n;c}^*: \cA_{dp}(cs_n) \hookrightarrow \cA_{dp}(c)
  \]
  are maps of cluster algebras, i.e.~they send clusters to clusters and exchange relations to exchange relations.
  In particular, $\eta_{1;c}^*$ takes cluster variables to cluster variables compatible with $x_{\omega_1;c}$ while~$\eta_{n;c}^*$ takes cluster variables to cluster variables compatible with $x_{-\omega_n;c}$.
\end{proposition}
\begin{proof}
  We begin by considering $\eta_{1;c}$.
  By definition, this map is dominant and to establish regularity it suffices to show that $\eta_{1;c}^*$ behaves as claimed.
  Since they are acyclic $\cA_{dp}(s_1c)$ and $\cA_{dp}(c)$ are lower bounds (cf.~\cite{BFZ05}), therefore it suffices to show that the map $\eta_{1;c}^*$ behaves as desired on their generators.
  The algebra $\cA_{dp}(c)$ is generated by $\{x_{\omega_i;c}\}_{i\in[1,n]}$, $\{x_{\omega_i;c}'\}_{i\in[1,n]}$, $\{z_{i;c}\}_{i\in[1,n]}$, and~$\{z_{\ol{\imath};c}\}_{i\in[1,n]}$ subject to the relations (for $j\in[1,n]$)
  \[
    x_{\omega_j;c}x_{\omega_j;c}' = \prod_{i>j} x_{\omega_i;c}^{-a_{ij}} + z_{j;c}z_{\ol{\jmath};c}\prod_{i<j} x_{\omega_i;c}^{-a_{ij}}.
  \]
  The algebra $\cA_{dp}(s_1c)$ is generated by $\{ x_{\omega_i;s_1 c}\}_{i\in[2,n]}$, $\{ x_{\omega_i;s_1 c}'\}_{i\in[2,n]}$, $\{ z_{i;s_1 c}\}_{i\in[2,n]}$, and $\{ z_{\ol{\imath};s_1 c}\}_{i\in[2,n]}$ subject to the relations (for $j\in[2,n]$)
  \[
     x_{\omega_j;s_1 c} x_{\omega_j;s_1 c}' = \prod_{i>j}  x_{\omega_i;s_1 c}^{-a_{ij}} +  z_{j;s_1 c} z_{\ol{\jmath};s_1 c}\prod_{1<i<j}  x_{\omega_i;s_1 c}^{-a_{ij}}.
  \]
  One way to embed $\cA_{dp}(s_1 c)$ inside $\cA_{dp}(c)$ is thus to choose the map given for~$j \in [2,n]$ by
  \[
    x_{\omega_j;s_1 c} \mapsto x_{\omega_j;c},\quad
    x_{\omega_j;s_1 c}' \mapsto x_{\omega_j;c}',\quad
    z_{j; s_1 c} \mapsto z_{j;c},\quad
    z_{\ol{\jmath}; s_1 c} \mapsto z_{\ol{\jmath};c} x_{\omega_1;c}^{-a_{1j}}.
  \]
  Using the formulas in \cite[Equation (3.10)]{RSW16}, this is easily seen to be the pullback along the map~$\eta_{1;c}:G^c\to G^{s_1c}_{\langle 1\rangle}$ in the claim.

  This shows that the map $\eta_{1;c}^*$ is injective.
  To conclude that it sends cluster variables to cluster variables it suffices to observe that it is a coefficient specialization.
  In particular, by \cite[Proposition 4.5]{YZ08}, it sends non-initial cluster variables of $\cA_{dp}(s_1 c)$ to cluster variables of $\cA_{dp}(c)$ multiplied by a monomial in the frozen variables.
  This monomial is computed by tropically evaluating the corresponding $F$-polynomials at $y_j=z_{j;c}z_{\ol{\jmath};c} x_{\omega_1;c}^{-a_{1j}}$ for $j\in[2,n]$.
  This time though, as opposed to the proof of Proposition \ref{prop:psi}, the specialized frozen variables have only positive exponents so that the tropical evaluation always yields $1$.

  To get the result for $\eta_{n;c}$ one proceeds in a similar way.
  This time though, one considers $\cA_{dp}(c)$ as the lower bound generated by $\{x_{-\omega_i;c}\}_{i\in[1,n]}$, $\{x_{-\omega_i;c}'\}_{i\in[1,n]}$, $\{z_{i;c}\}_{i\in[1,n]}$, and $\{z_{\ol{\imath};c}\}_{i\in[1,n]}$ subject to the relations (for $j\in[1,n]$)
  \[
    x_{-\omega_j;c}x_{-\omega_j;c}' =  z_{j;c}z_{\ol{\jmath};c} \prod_{i>j} x_{-\omega_i;c}^{-a_{ij}} + \prod_{i<j} x_{-\omega_i;c}^{-a_{ij}}.\qedhere
  \]
\end{proof}

\begin{remark}
  Both maps $\eta_{1;c}^*$ and $\eta_{n;c}^*$ arise from freezing cluster variables in $\cA_{dp}(c)$; the former corresponds to freezing $x_{\omega_1;c}$ and the latter to freezing $x_{-\omega_n;c}$.
  The same argument in the proof of Lemma~\ref{lem:canopy} shows that both maps are well behaved with respect to $\bfg$-vectors.
  In particular, consider the weight lattice $P_{\langle 1 \rangle}$ of $G_{\langle 1 \rangle}$ embedded into $P$ as the set of weight whose $\omega_1$-coordinate is $0$.
  Then the map~$\eta_{1;c}^*$ sends the cluster variable with $\bfg$-vector $\lambda\in P_{\langle 1\rangle}$ to the cluster variable with the same~$\bfg$-vector $\lambda \in P$.
  A similar consideration holds for $\eta_{n;c}^*$.
\end{remark}

\begin{lemma}
  \label{lem:rankredbygvec}
  Let $\lambda = \sum_{i=1}^n \lambda_i \omega_i \in P$ be the $\bfg$-vector of a cluster monomial $x_{\lambda;c}$ in $\cA_{dp}(c)$.
  Write~$\lambda_{\langle 1\rangle} = \sum_{i=2}^n \lambda_i \omega_i\in P_{\langle 1\rangle}$ and $\lambda_{\langle n\rangle} = \sum_{i=1}^{n-1} \lambda_i \omega_i\in P_{\langle n\rangle}$.
  If $\lambda_1 \geq 0$ then
  \[
    x_{\lambda;c} = (x_{\omega_1;c})^{\lambda_1} \eta^*_{1;c}(x_{\lambda_{\langle 1\rangle}; s_1c}),
  \]
  and if $\lambda_n \leq 0$ then
  \[
    x_{\lambda;c} = (x_{-\omega_n;c})^{-\lambda_n} \eta^*_{n;c}(x_{\lambda_{\langle n\rangle};cs_n}).
  \]
\end{lemma}
\begin{proof}
  Again we deal only with the first claim, the other being obtained in a similar way.
  In view of Lemma \ref{lem:canopy}, the cluster monomial $x_{\lambda;c}$ factors as $(x_{\omega_1;c})^{\lambda_1} x_{\lambda-\lambda_1\omega_1;c}$.
  The fact that $x_{\lambda-\lambda_1\omega_1;c}$ is the image of $x_{\lambda_{\langle 1\rangle}, s_1c}$ via $\eta^*_{1;c}$ follows then from \cite[Proposition 12]{RuSt17}.
\end{proof}

\begin{proposition}
  \label{prop:minorrankinduction}
  Let $V$ be a weight representation of $G$ for which $\lambda = \sum_{i=1}^n \lambda_i \omega_i$ is extremal, and let $\Delta_{V, \lambda}$ be the minor defined by some choice of $v_\lambda$, $\pi_\lambda$.
  If $\lambda_1 \geq 0$ then
  \[
    \Delta_{V, \lambda}\Big|_{\dbc{c}}
    =
    \left(\Delta_{\omega_1}\Big|_{\dbc{c}}\right)^{\lambda_1}
    \eta_{1;c}^* \left(\Delta_{V_{\langle 1\rangle}, \lambda_{\langle 1\rangle}}\Big|_{\dbc{s_1c}[c^{-1}s_1]_{\langle 1 \rangle}}\right),
  \]
  where $V_{\langle 1 \rangle}$ denotes $V$ as a $G_{\langle 1\rangle}$-representation, $\lambda_{\langle 1\rangle} = \sum_{i=2}^n \lambda_i \omega_i$ is the restriction of $\lambda$ to $H_{\langle 1 \rangle}$, and~$\Delta_{V_{\langle 1\rangle}, \lambda_{\langle 1\rangle}}$ is defined by the same $v_\lambda$, $\pi_\lambda$ used to define $\Delta_{V, \lambda}$.
  Similarly, if $\lambda_n \leq 0$ then
  \[
    \Delta_{V, \lambda}\Big|_{\dbc{c}}
    =
    \left(\Delta_{-\omega_n}\Big|_{\dbc{c}}\right)^{-\lambda_n}
    \eta_{n;c}^* \left(\Delta_{V_{\langle n\rangle}, \lambda_{\langle n\rangle}}\Big|_{\dbc{cs_n}[s_nc^{-1}]_{\langle n \rangle}}\right),
  \]
  the notation being defined analogously.
\end{proposition}
\begin{proof}
  Assume at first that $\lambda_1\geq0$ and compute on the generic element
  \[
    g=x_{\ol{1}}(t_{\ol{1}}) \cdots x_{\ol{n}}(t_{\ol{n}}) h_1^{\alpha_1^\vee} \cdots h_n^{\alpha_n^\vee} x_n(t_n) \cdots x_1(t_1).
  \]
  By Lemma \ref{lemma:one extra step}, we have
  \begin{align*}
    \Delta_{V,\lambda}(g)
    =&
    \Delta_{V,\lambda}\Big(x_{\ol{1}}(t_{\ol{1}}) x_{\ol{2}}(t_{\ol{2}}) \cdots x_{\ol{n}}(t_{\ol{n}}) h_1^{\alpha_1^\vee}\cdots h_n^{\alpha_n^\vee} x_n(t_n) \cdots x_2(t_2) x_1(t_1)\Big)
    \\=&
    \Delta_{V,\lambda}\Big(x_{\ol{2}}(t_{\ol{2}}) \cdots x_{\ol{n}}(t_{\ol{n}}) h_1^{\alpha_1^\vee}\cdots h_n^{\alpha_n^\vee} x_n(t_n) \cdots x_2(t_2)\Big)
    .
    \intertext{Then using the commutation rules together with the observation that $\Delta_{V,\lambda}(g' h)=h^\lambda\Delta_{V,\lambda}(g')$ for any $h\in H$, $g'\in G$, we get}
    \Delta_{V,\lambda}(g) =&
    \Delta_{V,\lambda}\Big(x_{\ol{2}}(t_{\ol{2}}) \cdots x_{\ol{n}}(t_{\ol{n}}) h_2^{\alpha_2^\vee}\cdots h_n^{\alpha_n^\vee} x_n(t_n h_1^{a_{1n}}) \cdots x_2(t_2 h_1^{a_{12}}) h_1^{\alpha_1^\vee}\Big)
    \\=&
    h_1^{\lambda_1}\Delta_{V,\lambda}\Big(x_{\ol{2}}(t_{\ol{2}}) \cdots x_{\ol{n}}(t_{\ol{n}}) h_2^{\alpha_2^\vee}\cdots h_n^{\alpha_n^\vee} x_n(t_n h_1^{a_{1n}}) \cdots x_2(t_2 h_1^{a_{12}})\Big)
    \\=&
    \Big(\Delta_{\omega_1}(g)\Big)^{\lambda_1}\Delta_{V,\lambda}\Big(x_{\ol{2}}(t_{\ol{2}}) \cdots x_{\ol{n}}(t_{\ol{n}}) h_2^{\alpha_2^\vee}\cdots h_n^{\alpha_n^\vee} x_n(t_n h_1^{a_{1n}}) \cdots x_2(t_2 h_1^{a_{12}})\Big).
    \intertext{%
      Note that $\Delta_{\omega_1}(g)=h_1$ follows directly from the factorization of $g$.
      By construction $\Delta_{V_{\langle 1\rangle},\lambda_{\langle 1\rangle}}$ is the restriction of $\Delta_{V,\lambda}$ along the canonical inclusion $G_{\langle 1 \rangle} \subset G$.
      Since the point on which the minor $\Delta_{V,\lambda}$ is computed in the last expression lies in $G_{\langle 1 \rangle} \subset G$, it further follows from the definition of $\eta_{1;c}$ that
    }
    \Delta_{V,\lambda}(g) =&
    \Big(\Delta_{\omega_1}(g)\Big)^{\lambda_1}\Delta_{V_{\langle 1\rangle},\lambda_{\langle 1 \rangle}}\Big(x_{\ol{2}}(t_{\ol{2}}) \cdots x_{\ol{n}}(t_{\ol{n}}) h_2^{\alpha_2^\vee}\cdots h_n^{\alpha_n^\vee} x_n(t_n h_1^{a_{1n}}) \cdots x_2(t_2 h_1^{a_{12}})\Big)
    \\=&
    \Big(\Delta_{\omega_1}(g)\Big)^{\lambda_1}\Delta_{V_{\langle 1\rangle},\lambda_{\langle 1\rangle}}\Big(\eta_{1;c}(g)\Big)
    \\=&
    \Big(\Delta_{\omega_1}(g)\Big)^{\lambda_1}\eta_{1;c}^*\Big(\Delta_{V_{\langle 1\rangle},\lambda_{\langle 1\rangle}}\Big)(g).
  \end{align*}

  The claim for $\lambda_n\leq0$ is obtained in exactly the same way from the generic factorization
  \[
    g=x_n(t_n) \cdots x_1(t_1) h_1^{\alpha_1^\vee} \cdots h_n^{\alpha_n^\vee} x_{\ol{1}}(t_{\ol{1}}) \cdots x_{\ol{n}}(t_{\ol{n}}).\qedhere
  \]
\end{proof}

\section{Cambrian cluster monomials}
\label{sec:bigcambrian}

Informally, our main result relies on the fact that  the seeds we consider, although not necessarily acyclic themselves, can be made to intersect the acyclic initial seed by mutating the latter only at sinks or sources.
Moreover, the same property holds inductively when freezing the cluster variables in this intersection.
In general, not every seed in an acyclic cluster algebra will satisfy this requirement, but all seeds do in finite and affine types.

Note that this condition is much weaker than asking for our seeds to be reachable from the initial one using only source-sink moves: source-sink mutations in the smaller rank cluster algebra can create cycles in the full algebra.
This is the reason why we are also able to discuss certain cyclic seeds by essentially understanding only source-sink moves.

These considerations are formalized by the (doubled) Cambrian fan of \cite{RS16,RS15}.
This is a simplicial fan obtained by gluing certain chambers in the Coxeter arrangement and which turns out to be a subfan of the $\bfg$-vector fan.
We will not need the full construction here so we briefly recall only the notions we will use; we refer the reader to the original papers starting from \cite{Rea07} for further details.

As before, fix a Coxeter element $c = s_1 \cdots s_n$ in the Weyl group $W$ of $\wh{G}$.
We say that a reflection~$s_i$ is \newword{initial} in $w\in W$ if there is a reduced expression for $w$ starting with $s_i$ (in other words, if the length of $s_i w$  is strictly less than the length of $w$).
The \newword{$c$-sortable elements} of $W$ are the elements satisfying the following recursive condition.
The identity  $e\in W$ is $c$-sortable.
Any other element $w\in W$ is $c$-sortable if either
\begin{itemize}
  \item
    $s_1$ is initial in $w$ and $s_1w$ is $s_1cs_1$-sortable, or
  \item
    $w$ lives in $W_{\langle 1 \rangle}$, the parabolic subgroup of $W$ not containing $s_1$, where it is $s_1c$-sortable.
\end{itemize}

Alternatively one can give the following non-recursive definition.
Write $c^\infty$ for the word obtained by concatenating infinitely many copies of our fixed reduced expression of $c$:
\[
  c^\infty = s_1\cdots s_n | s_1 \cdots s_n | s_1 \cdots s_n | \cdots.
\]
The \newword{$c$-sorting word} for $w\in W$ is the lexicographically leftmost subword of $c^\infty$ that is a reduced expression of $w$.
An element $w$ is $c$-sortable if and only if its $c$-sorting word is such that whenever a reflection $s_j$ is skipped from a copy of $c$ in $c^\infty$, it is also skipped in all copies of $c$ to its right.
Note that, even though the $c$-sorting word of $w$ depends on the fixed reduced expression of $c$, the definition is well-posed since any two such expressions differ only by transpositions of commuting reflections.

To each $c$-sortable element $w$, one can associate a collection~$\cl_c(w)$ of $n$ roots as follows.
For each~$i\in[1,n]$, denote by $w^{(i)}$ the prefix of the $c$-sorting word of $w$ obtained by dropping every simple reflection to the right of the rightmost occurrence of $s_i$.
If $w \in W_{\langle i \rangle}$, the prefix $w^{(i)}$ is the identity.
The set $\cl_c(w)$ is then
\begin{equation}
  \label{equation:non-recursive cl}
  \cl_c(w)
  :=
  \big\{-w^{(i)}(\alpha_i)\big\}_{i\in[1,n]},
\end{equation}
where again the consistency of the definition is guaranteed by the fact that reduced expressions for $c$ are connected by commutations.

Note that, since the $c$-sorting word of $w$ is a reduced expression, for each $i\in[1,n]$ only two cases are possible for the roots in $\cl_c(w)$: either $w$ is in $W_{\langle i \rangle}$ so that $-w^{(i)}(\alpha_i)=-\alpha_i$, or $-w^{(i)}(\alpha_i)$ is a positive root.

Recall that, given a vector $\beta$ in the root lattice, we denote by $[\beta:\alpha_i]$ its $i$-th coordinate when written in the basis of simple roots.
Let $[\beta:\alpha_i]_+$ be the maximum of $[\beta:\alpha_i]$ and $0$.
Write $\nu_c$ for the piecewise linear map
\begin{equation}
  \label{eq:nu}
  \nu_c(\beta)
  =
  - \sum_{i=1}^n \left( [\beta:\alpha_i] + \sum_{j < i} a_{ij} [\beta:\alpha_j]_+ \right)\omega_i
\end{equation}
from the root lattice to the weight lattice; this map, introduced in \cite{RS17}, is the one used to pass from $\bfd$-vectors to $\bfg$-vectors in the proof of Lemma~\ref{lem:canopy} (cf.~\cite[Proposition 9]{RuSt17}).
When~$\beta$ has only non-negative coordinates, the matrix representing $\nu_c$ in the bases of simple roots and fundamental weights is lower triangular, whereas it is diagonal if $\beta$ has only non-positive coordinates.

For any $c$-sortable element $w$, denote by $\Cone_c(w)$ the non-negative span of the vectors $\nu_c\left(\cl_c(w)\right)$.
Write $\cF_c$ for the collection of cones of the form $\Cone_c(w)$ for $w$ a $c$-sortable element, together with all faces of such cones.
Finally, let $\cD\cF_c$ be the union of $\cF_c$ and $-\cF_{c^{-1}}:=\left\{ -C \, |\, C \in \cF_{c^{-1}}\right\}$.

\begin{theorem}[{\cite[Corollary 1.3, Theorem 2.3, and Remark 3.25]{RS15}}]
  \label{thm:cambrian}
  The collection $\cD\cF_c$ is a simplicial fan and is a subfan of the $\bfg$-vector fan of $\cA_{dp}(c)$.
  The subcollections $\cF_c$ and $-\cF_{c^{-1}}$ are subfans which respectively cover the Tits cone and its opposite.
  In finite and affine types, $\cD\cF_c$ coincides with the $\bfg$-vector fan of $\cA_{dp}(c)$.
\end{theorem}

Following \cite{RS15} $\cF_c$ is referred to as the \newword{$c$-Cambrian fan}, $-\cF_{c^{-1}}$ as the \newword{opposite $c$-Cambrian fan}, and $\cD\cF_c$ as the \newword{doubled $c$-Cambrian fan}.

Our strategy to prove Theorem \ref{thm:mainthm2} is to lift, using the doubled $c$-Cambrian fan $\cD\cF_c$, the recursions defining $c$-sortable and $c^{-1}$-sortable elements to the level of $\cA_{dp}(c)$.
Observe that, since $\cF_c$ covers the Tits cone, it contains the cone corresponding to the initial seed of $\cA_{dp}(c)$; this is the cone~$\Cone_c(e)$.
Similarly, the cone $-\Cone_{c^{-1}}(e)$ in $-\cF_{c^{-1}}$ corresponds to the ``opposite initial seed'', i.e.~the seed whose cluster is $\{ x_{-\omega_i;c}\}_{i\in[1,n]}$; this is obtained from the initial one by mutating along the sequence~$(n,n-1,\dots,1)$.

\begin{remark}
  \label{rk:cambprop}
  Implicit in Theorem~\ref{thm:cambrian} is the fact that the vectors $\nu_c\left(\cl_c(w)\right)$ are $\bfg$-vectors of cluster variables, and that the set of cluster variables whose $\bfg$-vectors lie in a fixed $\Cone_c(w)$ form a cluster.
  In particular, if $w$ is in $W_{\langle i \rangle}$ then the seed corresponding to~$w$ contains the initial cluster variable $x_{\omega_i;c}$.
  Note that, contrary to one's intuition, not every seed containing $x_{\omega_i;c}$ has $\bfg$-vector cone belonging to~$\cF_c$ because, when $W_{\langle i \rangle}$ is not finite, $\omega_i$ is on the boundary of the Tits cone (cf.~\cite[Lemma~2.86 and Exercise 2.90]{AB08}).

  The same reasoning applies to $-\cF_{c^{-1}}$: the $\bfg$-vectors of the cluster variables in the seed corresponding to a $c^{-1}$-sortable element~$w$ are the vectors $-\nu_{c^{-1}}\left(\cl_{c^{-1}}(w)\right)$.
\end{remark}

Combining the above observations with Lemma \ref{lem:canopy}, we get the following corollary.
\begin{corollary}
  \label{cor:canopy}
  Suppose $x_{\lambda;c}$ is a cluster monomial in $\cA_{dp}(c)$ whose $\bfg$-vector $\lambda = \sum_{i=1}^n \lambda_i\omega_i$ is contained in a cone $\Cone_c(w)$ of the $c$-Cambrian fan $\cF_c$.
  Then either $s_1$ is initial in $w$ and~$\lambda_1\leq 0$, or $w\in W_{\langle 1\rangle}$ and $\lambda_1 \geq 0$.
  Similarly, if $\lambda$ is contained in a cone $-\Cone_{c^{-1}}(w)$ of the opposite~$c$-Cambrian fan $-\cF_{c^{-1}}$, then either $s_n$ is initial in $w$ and $\lambda_n\geq 0$, or $w\in W_{\langle n\rangle}$ and $\lambda_n \leq 0$.
\end{corollary}
\begin{proof}
  We show the claim for the $c$-Cambrian fan, the other being obtained in the same way.
  The only case that is not obvious is when $w\in W_{\langle 1\rangle}$.
  Observe that, in this situation, any root $\beta$ in the set~$\cl_c(w)$ different from $-\alpha_1$ has $[\beta:\alpha_1]=0$.
  Indeed, $\beta=-w^{(i)}\alpha_i$ is obtained from a simple root~$\alpha_i\neq\alpha_1$ by the action of an element in $W_{\langle 1\rangle}$.
  In particular, all the weights in $\nu_c\left(\cl_c(w)\right)$ have first coordinate equal to 0 with the exception of $\nu_c(-\alpha_1) =\omega_1$.
\end{proof}

We conclude this section by describing how certain cones in the doubled $c$-Cambrian fan $\cD\cF_c$ change when $c$ is conjugated by $s_1$ or $s_n$.
By symmetry, we only need to understand the case $\cF_c$.
\begin{lemma}
  \label{lemma:cone reflection}
  Suppose $w$ is $c$-sortable and $s_1$ is initial in $w$.
  Then
  \[
    s_1\Big(\Cone_c(w)\Big) = \Cone_{s_1cs_1}(s_1w).
  \]
\end{lemma}
\begin{proof}
  It suffices to show that $s_1$ maps the rays of $\Cone_c(w)$ bijectively onto the rays of $\Cone_{s_1cs_1}(s_1w)$; we will therefore consider all possible roots $\beta\in\cl_c(w)$.
  Begin by observing that $-\alpha_1\not \in \cl_c(w)$, since $s_1$ is initial in $w$ and in view of Equation \ref{equation:non-recursive cl}.
  We dispatch first of two special cases.

  If $\beta=-\alpha_j$ for some $j\neq 1$, then $w\in W_{\langle j \rangle}$.
  Therefore~$s_1w \in W_{\langle j \rangle}$ and $-\alpha_j \in \cl_{s_1cs_1}(s_1w)$.
  We get
  \[
    s_1\big(\nu_c(-\alpha_j)\big)
    =
    s_1(\omega_j)
    =
    \omega_j
    =
    \nu_{s_1cs_1}(-\alpha_j).
  \]

  If $\beta=\alpha_1$, the reflection $s_1$ can only appear once in the $c$-sorting word of $w$: it is the initial letter and $s_1w\in W_{\langle 1 \rangle}$ so that $-\alpha_1\in\cl_{s_1cs_1}(s_1w)$.
  We get
  \[
    s_1\big(\nu_c(\alpha_1)\big)
    =
    s_1\left(-\omega_1 - \sum_{i=2}^n a_{i1} \omega_i\right)
    =
    -\omega_1 + \alpha_1 - \sum_{i=2}^n a_{i1} \omega_i
    =
    \omega_1
    =
    \nu_{s_1cs_1}(-\alpha_1).
  \]

  Now let $\beta$ be any other (necessarily positive) root in $\cl_c(w)$.
  Since $\beta\neq\alpha_1$, the root $s_1(\beta)$ is positive and an element of $\cl_{s_1cs_1}(s_1w)$.
  Using the expression \eqref{eq:nu}, we then compute
  \begin{align*}
    s_1\big(\nu_c(\beta)\big)
    &=
    \nu_c(\beta) + [\beta:\alpha_1]\alpha_1
    \\&=
    - \sum_{i=1}^n\left( [\beta:\alpha_i] + \sum_{j< i}a_{ij}[\beta: \alpha_j]\right)\omega_i +[\beta:\alpha_1]\sum_{i=1}^na_{i1}\omega_i
    \\&=
    - \sum_{i=1}^n\left( [\beta:\alpha_i] - a_{i1}[\beta:\alpha_1] + \sum_{j< i}a_{ij}[\beta: \alpha_j]\right)\omega_i
    \\&=
    [\beta:\alpha_1]\omega_1 - \sum_{i=2}^n\left( [\beta:\alpha_i] + \sum_{1\neq j< i}a_{ij}[\beta: \alpha_j]\right)\omega_i.
  \end{align*}
  By adding and subtracting $\sum_{j=1}^na_{1j}[\beta:\alpha_j]\omega_1$, we get
  \begin{align*}
    &=
    - \left([\beta:\alpha_1] + \sum_{j \neq 1}a_{1j}[\beta: \alpha_j]\right)\omega_1  + \sum_{j=1}^na_{1j}[\beta:\alpha_j]\omega_1
    \\&\qquad\qquad\qquad\qquad - \sum_{i=2}^n\left( [\beta:\alpha_i] + \sum_{1\neq j< i}a_{ij}[\beta: \alpha_j]\right)\omega_i
    \\&=
    \nu_{s_1cs_1}\left(\beta-\sum_{j=1}^na_{1j}[\beta:\alpha_j]\alpha_1\right)
    \\&=
    \nu_{s_1cs_1}\big(s_1(\beta)\big).\qedhere
  \end{align*}
\end{proof}

We now have all the required tools to prove our main result.
\begin{theorem}
  \label{th:main}
  Let $\lambda = \sum_{i=1}^n \lambda_i \omega_i \in \cD\cF_c$ be a weight in the doubled $c$-Cambrian fan, and let~$V$ be a weight representation of $G$ for which $\lambda$ is extremal.
  Then the cluster monomial $x_{\lambda; c}$ of~$\cA_{dp}(c) \cong \kk[\dbc{c}]$ coincides with the restriction to $\dbc{c}$ of the principal generalized minor $\Delta_{V, \lambda}$ for any choice of $v_\lambda$ and~$\pi_\lambda$.
\end{theorem}
\begin{proof}
  There are two (possibly) overlapping cases to consider, that where $\lambda$ lies in $\cF_c$ and that where $\lambda$ lies in~$-\cF_{c^{-1}}$.
  If $\lambda$ is in $\cF_c$, there exists a $c$-sortable element $w$ such that $\lambda\in \Cone_c(w)$.
  In general, there could be many such elements (this happens precisely when $\lambda$ is not in the interior of a maximal cone of $\cF_c$), and for definiteness we choose $w\in W$ to be of minimal length in this case.
  We proceed by a double induction on the rank of $G$ and the length of $w$.
  The base case, $G=SL_2$ and $w=e$, follows immediately from Lemma \ref{lem:rankredbygvec} and Proposition \ref{prop:minorrankinduction} by observing that in both results the contribution coming from $\eta_{1;c}^*$ is just $1$.

  The definition of $c$-sortability gives us two cases: either $s_1$ is initial in $w$ or $w\in W_{\langle 1 \rangle}$.
  If $s_1$ is initial in $w$, then by Remark \ref{rk:cambprop} the weight $\omega_1$ is not an element of $\nu_c\left(\cl_c(w)\right)$.
  By Proposition~\ref{prop:psi}, we have
  \[
    x_{\lambda;c}
    =
    \prod_{\rho\in \nu_c\left(\cl_c(w)\right)} x_{\rho;c}^{a_\rho}
    =
    \prod_{\rho\in \nu_c\left(\cl_c(w)\right)} \psi_{1;c}^*\left(x_{s_1\rho;\muoc}\right)^{a_\rho}
    =
    \psi_{1;c}^*\left(x_{s_1\lambda;\muoc}\right).
  \]
  Lemma \ref{lemma:cone reflection} says that $s_1\lambda$ belongs to $\Cone_{s_1cs_1}(s_1w)$.
  Since the length of $s_1w$ is strictly less than the length of $w$, it follows from our inductive hypothesis that
  \[
    x_{s_1\lambda;\muoc}=\Delta_{V,s_1\lambda}\Big|_{\dbc{\muoc}}.
  \]
  Finally, Corollary \ref{cor:canopy} says that $\lambda_1\leq 0$, so we can apply Proposition \ref{prop:psiminors} to get
  \[
    \psi_{1;c}^*\left(\Delta_{V,s_1\lambda}\Big|_{\dbc{\muoc}}\right)
    =
    \Delta_{V,\lambda}\Big|_{\dbc{c}}.
  \]
  Putting these equalities together, we see that
  \[
    x_{\lambda;c}
    =
    \psi_{1;c}^*\left(x_{s_1\lambda;\muoc}\right)
    =
    \psi_{1;c}^*\left(\Delta_{V,s_1\lambda}\Big|_{\dbc{\muoc}}\right)
    =
    \Delta_{V,\lambda}\Big|_{\dbc{c}}
  \]
  as desired.

  Suppose now that $w\in W_{\langle 1 \rangle}$.
  By Corollary \ref{cor:canopy}, we have $\lambda_1 \geq 0$ and thus Lemma \ref{lem:rankredbygvec} gives
  \[
    x_{\lambda; c} = (x_{\omega_1; c})^{\lambda_1} \eta^*_{1;c}(x_{\lambda_{\langle 1\rangle}; s_1c}),
  \]
  where $\lambda_{\langle 1\rangle} = \sum_{i=2}^n \lambda_i \omega_i \in \cF_{s_1c}$.
  On the other hand, by Proposition \ref{prop:minorrankinduction}, we have
  \[
    \Delta_{V, \lambda}\Big|_{\dbc{c}}
    =
    \left(\Delta_{\omega_1}\Big|_{\dbc{c}}\right)^{\lambda_1}
    \eta_{1;c}^* \left(\Delta_{V_{\langle 1 \rangle}, \lambda_{\langle 1 \rangle}}\Big|_{\dbc{s_1c}_{\langle 1\rangle}}\right).
  \]
  By induction on rank, we have $x_{\lambda_{\langle 1 \rangle}; s_1c} = \Delta_{V_{\langle 1 \rangle}, \lambda_{\langle 1 \rangle}}\Big|_{\dbc{s_1c}_{\langle 1\rangle}}$.
  Finally, since $x_{\omega_1;c}~=~\Delta_{\omega_1}\Big|_{\dbc{c}}$, we conclude that
  \[
    x_{\lambda; c}
    =
    (x_{\omega_1; c})^{\lambda_1} \eta^*_{1;c}(x_{\lambda_{\langle 1 \rangle}; s_1c})
    =
    \left(\Delta_{\omega_1}\Big|_{\dbc{c}}\right)^{\lambda_1}  \eta_{1;c}^* \left(\Delta_{V_{\langle 1 \rangle}, \lambda_{\langle 1 \rangle}}\Big|_{\dbc{s_1c}_{\langle 1\rangle}}\right)
    =
    \Delta_{V, \lambda}\Big|_{\dbc{c}}.
  \]
  The other case, when $\lambda$ is in the opposite $c$-Cambrian fan, is established in exactly the same way.
\end{proof}

\begin{remark} 
  Theorem \ref{thm:mainthm2} can be extended a little bit further as follows.
  In general, beyond affine types, the seeds of $\cA_{dp}(c)$ obtained from the initial one by mutating along the sequence $(n,n-1,\dots,1)$ need not have $\bfg$-vector cone belonging to $\cD\cF_c$.
  Indeed, if this were the case, any doubled Cambian framework would be connected in the sense of \cite{RS15}.
  Nonetheless, the same argument we have used in the rank induction suffices to establish the result also for cluster monomials supported on these seeds, since they contain only cluster variables of the form $x_{\omega_i; c}$ or~$x_{-\omega_i; c}$.
\end{remark}

\section{Canonical bases in the case $A_1^{(1)}$}
\label{sec:bases}

Theorem~\ref{th:main} shows that for affine types all elements of the partial canonical basis of cluster monomials in $\cA_{dp}(c)$ are computable as generalized minors.
There are several completions of this partial basis to a full basis of $\cA_{dp}(c)$ that appear in the literature.
Some of these include:
\begin{itemize}
  \item the atomic basis \cite{CI12,DT13} for cluster algebras of finite and affine type;
  \item the greedy basis \cite{SZ04,LLZ14} for cluster algebras of rank 2;
  \item the generic basis \cite{Dup11,Dup12,Pla13} for skew-symmetric cluster algebras;
  \item the bracelets, bangles, and bands bases \cite{MSW13,T14} for cluster algebras from surfaces;
  \item the triangular basis \cite{BZ14} for acyclic cluster algebras;
  \item the theta basis \cite{GHKK14} for arbitrary cluster algebras.
\end{itemize}

While these bases have been constructed in diverse settings, in specific cases subsets of them can be shown to coincide (though not necessarily in a straightforward way, cf.~\cite{CGMMRSW15}).
For the cluster algebra of type $A_1^{(1)}\!$, for example, only three distinct bases appear in the literature.
In this case, the atomic, greedy, bracelets, and theta bases all coincide, the generic basis coincides with the bangles basis, and finally the bands basis and triangular basis coincide.

In any of these bases, the elements which are not cluster monomials are exactly those whose $\bfg$-vector is of the form~$n\omega^\circ$, where $\omega^\circ = -\omega_1+\omega_2$ and $n$ is a positive integer.
Note that, for consistency with the preceding sections, we index the nodes of the associated Dynkin diagram of type $A_1^{(1)}$ by $\{1,2\}$ rather than the standard $\{0,1\}$; by convention, we identify $1$ as the affine node.
We have the following explicit formulas for these basis elements in the presence of doubled principal coefficients.

\vspace{8pt}
\noindent{\bf Greedy basis} \cite[Theorem 5.4]{CZ06}:
  \begin{align*}
    x_{n\omega^\circ;c}^{gr}
    :=
    x_{\omega_1;c}^{-n} x_{\omega_2;c}^{-n}
    \left(
      \sum_{0 < k \leq m \leq n}
      \frac{n}{k} {m-1\choose k-1} {n-m+k-1\choose n-m}
      \, z_{\ol{1};c}^m z_{1;c}^m z_{\ol{2};c}^{m-k} z_{2;c}^{m-k} x_{\omega_1;c}^{2(m-k)} x_{\omega_2;c}^{2(n-m)}
      \right. \\ \left.
      \vphantom{\sum_{\substack{0\leq m \leq n\\ 1 \leq k \leq m}}}
      +
      x_{\omega_2;c}^{2n}
      +
      z_{\ol{1};c}^n z_{1;c}^n z_{\ol{2};c}^{n} z_{2;c}^{n} x_{\omega_1;c}^{2n}
    \right)
  \end{align*}

\noindent{\bf Triangular basis} \cite[Theorem 5.2]{CZ06}:
  \[
    x_{n\omega^\circ;c}^{tr}
    :=
    x_{\omega_1;c}^{-n} x_{\omega_2;c}^{-n}
    \sum_{0 \leq k \leq m \leq n}
    {m\choose k} {n-m+k\choose n-m}
    \, z_{\ol{1};c}^m z_{1;c}^m z_{\ol{2};c}^{m-k} z_{2;c}^{m-k} x_{\omega_1;c}^{2(m-k)} x_{\omega_2;c}^{2(n-m)}
  \]

\noindent{\bf Generic basis}:
  \[
    x_{n\omega^\circ;c}^{ge}
    :=
    x_{\omega_1;c}^{-n} x_{\omega_2;c}^{-n}
    \sum_{0 \leq k \leq m \leq n}
    {m\choose k} {n\choose m}
    \, z_{\ol{1};c}^m z_{1;c}^m z_{\ol{2};c}^{m-k} z_{2;c}^{m-k} x_{\omega_1;c}^{2(m-k)} x_{\omega_2;c}^{2(n-m)}
  \]

The first two formulas are obtained from those in \cite{CZ06} by inserting frozen variables so that the elements become homogeneous of the appropriate degree and are invariant when interchanging $z_{i;c}$ and $z_{\ol{\imath};c}$; the last is obtained directly by expanding the defining formula
\[
  x_{n\omega^\circ;c}^{ge}
  :=
  \left(x_{\omega^\circ;c}^{ge}\right)^n
  =
  \Big(x_{\omega_1;c}^{-1} x_{\omega_2;c}^{-1}  \big(z_{\ol{1};c} z_{1;c} z_{\ol{2};c} z_{2;c} x_{\omega_1;c}^{2} + z_{\ol{1};c} z_{1;c} + x_{\omega_2;c}^{2} \big) \Big)^n.
\]

In this section, we show that all elements in each of these bases of $\cA_{dp}(c)$ are restrictions of principal generalized minors of irreducible level-zero representations of $\wh{G} = \wh{LSL}_2$.
Such representations have continuous parameters and to obtain each of the different bases we need only choose different parameters.

\begin{remark}
  Outside of finite type, it need not be the case that all weight representations of $G$ extend to representations of the larger group $\wh{G}$.
  For example, in untwisted affine types, the evaluation representations of the centrally extended loop group $G = \widetilde{LG^\circ}$ do not admit an intertwining action of the group of loop rotations, hence do not extend to representations of $\wh{G} = \wh{LG^\circ}$.
  Thus the statement that the above basis elements are restrictions of minors of $\wh{LSL}_2$-representations is stronger than the statement that they are restrictions minors of $\widetilde{LSL}_2$-representations.
  Conversely, the statement in Theorem \ref{thm:mainthm2} that a Cambrian cluster monomial $x_{\lambda;c}$ is equal to the restriction of the minor $\Delta_{V,\lambda}$ for any $G$-representation $V$ is stronger than the corresponding statement with $\wh{G}$ in place of $G$.
\end{remark}

\begin{remark}
  \label{rem:quiver bases}
  These three bases for $A_1^{(1)}$ have all been given quiver-theoretic interpretations: the generic basis elements are defined inherently using generic regular quiver representations \cite{Dup11}, the bands/triangular basis elements were computed in \cite{CZ06} using indecomposable regular quiver representations, and the greedy/bracelets/theta basis elements were computed in \cite{CIE11} using only the smooth locus of quiver Grassmannians for indecomposable regular quiver representations.
\end{remark}

Let $V \cong \kk^2$ be the vector representation of $SL_2$ and let $\bfa=(a_1,\dots,a_n)$ be an $n$-tuple of elements in $\kk^\times$.
We denote by $V(\bfa)$ the irreducible representation of~$\wh{G} \cong \wh{LSL}_2$ whose underlying vector space is $\kk[u^{\pm 1}]\otimes V^{\otimes n}$ with the action given by
\begin{gather*}
  x_1(t) \Big( u^m \otimes v_1 \otimes \cdots \otimes v_n \Big) := u^{m+1} \otimes \left( \begin{array}{cc}1 & 0 \\ a_1 t & 1\end{array}\right) v_1  \otimes \cdots \otimes \left( \begin{array}{cc}1 & 0 \\ a_n t & 1\end{array}\right) v_n;
  \\
  x_{\ol{1}}(t) \Big( u^m \otimes v_1 \otimes \cdots \otimes v_n \Big) := u^{m-1} \otimes \left( \begin{array}{cc}1 & a_1^{-1} t \\ 0 & 1\end{array}\right) v_1  \otimes \cdots \otimes \left( \begin{array}{cc}1 & a_n^{-1} t \\ 0 & 1\end{array}\right) v_n;
  \\
  x_2(t) \Big( u^m \otimes v_1 \otimes \cdots \otimes v_n \Big) := u^m \otimes \left( \begin{array}{cc}1 & t \\ 0 & 1\end{array}\right) v_1  \otimes \cdots \otimes \left( \begin{array}{cc}1 & t \\ 0 & 1\end{array}\right) v_n;
  \\
  x_{\ol{2}}(t) \Big( u^m \otimes v_1 \otimes \cdots \otimes v_n \Big) := u^m \otimes \left( \begin{array}{cc}1 & 0 \\ t & 1\end{array}\right) v_1  \otimes \cdots \otimes \left( \begin{array}{cc}1 & 0 \\ t & 1\end{array}\right) v_n.
\end{gather*}
By \cite{Cha86, CP86}, any irreducible level-zero weight representation of $\wh{LSL}_2$ with finite-dimensional weight spaces is of this form.

Let $v \in V$ be a highest weight vector for $SL_2$ and set $v_{n \omega^\circ} := v \otimes \cdots \otimes v \in V(\bfa)$.
We then let $\Delta_{V(\bfa),n\omega^\circ}$ be the minor defined by choosing $\pi_{n \omega^\circ}$ so that $\pi_{n \omega^\circ}(u^m v_{n\omega^\circ}) = 0$ for $m \neq 0$.

\begin{proposition}
  \label{prop:imaginary minor}
  For any $n$-tuple $\bfa=(a_1,\dots,a_n)\in(\kk^\times)^n$, the value of $\Delta_{V(\bfa),n\omega^\circ}$ on
  \[
    g:=x_{\ol{1}}(t_{\ol{1}}) x_{\ol{2}}(t_{\ol{2}}) h_1^{\alpha_1^\vee}h_2^{\alpha_2^\vee} x_2(t_2) x_1(t_1)
  \]
  is equal to
  \[
    \Delta_{V(\bfa),n\omega^\circ}(g)
    =
    \sum_{0 \leq k \leq m \leq n} d_\bfa(m,k) \, t_{\ol{1}}^m t_{\ol{2}}^{m-k} h_1^{2k-n} h_2^{n-2k} t_2^{m-k} t_1^m,
  \]
  where
  \[
    d_\bfa(m,k)=\sum\limits_{r=0}^m {m-r\choose k} {n-2r\choose m-r} S_{\bfa,r}
    \qquad
    \text{and}
    \qquad
    S_{\bfa,r}=\sum\limits_{\substack{I,J\subseteq[1,n]\\|I|=r=|J|\\ I\cap J=\varnothing}}\frac{\prod_{i\in I} a_i}{\prod_{j\in J} a_j}.
  \]
  By the isomorphism of Proposition~\ref{prop:coordring}, this minor is identified with the following element of~$\cA_{dp}(c)$:
  \[
    x_{n\omega^\circ;c}^{\bfa}
    :=
    x_{\omega_1;c}^{-n} x_{\omega_2;c}^{-n}
    \sum_{0 \leq k \leq m \leq n} d_\bfa(m,k) \, z_{\ol{1};c}^m z_{1;c}^m z_{\ol{2};c}^{m-k} z_{2;c}^{m-k} x_{\omega_1;c}^{2(m-k)} x_{\omega_2;c}^{2(n-m)}.
  \]
\end{proposition}
\begin{proof}
  The first claim is a simple computation parsing the action of~$g$ on $v_{n \omega^\circ}$ using the definition.
  The second claim then follows directly by applying the isomorphism of Proposition~\ref{prop:coordring}.
\end{proof}

\begin{remark}
  Independent of the chosen parameters $\bfa=(a_1,\ldots,a_n)$, we have that $d_\bfa(0,0)=1$ and $d_\bfa(n,0)=1$.
  The first equality is obvious.
  To see the other, observe that ${n-2r\choose n-r}$ is non-zero only if $n-2r \geq n-r$, i.e.~if $r=0$.
  More generally, by subsuming the summation over $r$ into the subsets $I,J \subseteq [1,n]$, the coefficient $d_\bfa(m,0)$ is given by the ratio
  \[
    d_\bfa(m,0)
    =
    \sum\limits_{\substack{I,J\subseteq[1,n]\\|I|=m=|J|}}\frac{\prod_{i\in I} a_i}{\prod_{j\in J} a_j}
    =
    \frac{1}{\prod\limits_{\ell\in [1,n]} a_\ell} \sum\limits_{\substack{I,J\subseteq[1,n]\\|I|=m=|J|}} \prod_{i\in I} a_i \prod_{j\not\in J} a_j
    =
    \frac{e_m(\bfa) e_{n-m}(\bfa)}{e_n(\bfa)},
  \]
  where $e_\ell(\bfa)$ is the elementary symmetric polynomial of degree $\ell$ evaluated on $\bfa$.
  In what follows, we will want to impose conditions on the coefficients $d_\bfa(m,0)$.
  Since the ring of symmetric polynomials is generated by the independent elements $e_\ell$, $\ell\in[1,n]$, to see that the conditions we want are satisfied it will be enough to give values to these rather than solving explicitly for the parameters $\bfa$.
\end{remark}

\begin{theorem}
  Choose a point $\bfa\in(\kk^\times)^n$ for each $n\geq1$.
  Then the elements $\{x_{n\omega^\circ;c}^{\bfa}\}_{n\geq1}$ together with all cluster monomials form a linear basis of $\cA_{dp}(c)$.
\end{theorem}
\begin{proof}
  First of all observe that, independent of the chosen points $\bfa$, we have $d_{\bfa}(n,n)=1$.
  In particular, each element in $\{x_{n\omega^\circ;c}^{\bfa}\}_{n\geq1}$ is pointed in the sense of \cite[Definition 2.8]{Rup13}.
  The set of cluster monomials, together with $\{x_{n\omega^\circ;c}^{\bfa}\}_{n\geq1}$, form a complete bounded collection of pointed elements and we can apply \cite[Proposition 2.9]{Rup13} to deduce the claim.
\end{proof}

The following result will feature prominently in the computations below.
\begin{lemma}
  \label{le:binomial identity}
  For $a,b,c\in\ZZ_{\ge0}$ with $0\le b,c\le a$ we have
  \[
		\sum\limits_{r=0}^{\min(b,c)}(-1)^r{b\choose r}{a-r\choose c-r}={a-b\choose c}.
	\]
\end{lemma}
\begin{proof}
  This can be proven using a combinatorial inclusion-exclusion argument that we leave as an exercise for the reader.
\end{proof}

Observe that, independent of the chosen point $\bfa\in(\kk^\times)^n$, we have $S_{\bfa,0}=1$ and $S_{\bfa,r}=0$ for any $r>\lfloor n/2\rfloor$ and thus the only terms contributing to the definition of $d_\bfa(m,k)$ are those for which~$r\leq \min(m,\lfloor n/2 \rfloor)$.
\begin{lemma}
  \label{le:a sums1}
  For $\bfa=(a_1,\dots,a_n)$, assume that $e_\ell(\bfa)=0$ whenever $\ell \in \big[1,\lfloor n/2\rfloor\big]$.
  Then for~$r\in\big[1 ,\lfloor n/2 \rfloor\big]$, we have $S_{\bfa,r}=(-1)^r\frac{n}{n-r}{n-r\choose r}$.
\end{lemma}
\begin{proof}
  The assumptions imply that $d_\bfa(m,0)=0$ for $m \in [1, n-1]$.
  Since the given formula satisfies $S_{\bfa,0}=1$, we may prove the result by induction on $r$.
  Suppose $S_{\bfa,r}=(-1)^r\frac{n}{n-r}{n-r\choose r}$ for~$0\le r < m \leq \lfloor n/2 \rfloor$.
  Then the condition $d_\bfa(m,0)=0$ gives
  \[
    S_{\bfa,m}=\sum\limits_{r=0}^{m-1} (-1)^{r+1}\frac{n}{n-r} {n-2r\choose m-r} {n-r\choose r}.
  \]
  But we have
  \[
    \frac{n}{n-r} {n-2r\choose m-r} {n-r\choose r}=\frac{n(n-r-1)!}{(m-r)!(n-m-r)!r!}=\frac{n}{n-m} {n-m\choose r} {n-r-1\choose m-r}
  \]
  so that
  \[
    S_{\bfa,m}=\frac{n}{n-m}\sum\limits_{r=0}^{m-1} (-1)^{r+1} {n-m\choose r} {n-r-1\choose m-r}.
  \]
  Applying Lemma~\ref{le:binomial identity} with $a=n-1$, $b=n-m$, and $c=m$ gives
  \[
    \sum\limits_{r=0}^{m-1} (-1)^{r+1} {n-m\choose r} {n-r-1\choose m-r} = (-1)^m {n-m \choose m}
  \]
  and thus
  \[
    S_{\bfa,m}=(-1)^m\frac{n}{n-m}{n-m\choose m}
  \]
  as desired.
\end{proof}

\begin{lemma}
  \label{le:a sums2}
  For $\bfa=(a_1,\dots,a_n)$, assume that $e_\ell(\bfa)=1$ whenever $\ell\in[1,n]$.
  Then we have~$S_{\bfa,r}=(-1)^r {n-r\choose r}$ for~$r\in\big[1,\lfloor n/2 \rfloor\big]$.
\end{lemma}
\begin{proof}
  The assumptions imply that $d_\bfa(m,0)=1$ for $m\in [1,n-1]$.
  We again note that $S_{\bfa,0}=1$ and work by induction on $r$.
  Suppose $S_{\bfa,r}=(-1)^r {n-r\choose r}$ for $0\le r<m \leq \lfloor n/2 \rfloor$.
  Now the condition $d_\bfa(m,0)=1$ gives
  \[
    S_{\bfa,m}-1
    =
    \sum\limits_{r=0}^{m-1} (-1)^{r+1} {n-2r\choose m-r} {n-r\choose r}
    =
    \sum\limits_{r=0}^{m-1} (-1)^{r+1}{n-m\choose r}{n-r\choose m-r},
  \]
  where we used the identity ${a-b\choose c}{a \choose b}={a-c \choose b}{a\choose c}$.
  Applying Lemma~\ref{le:binomial identity} with $a=n$, $b=n-m$, and~$c=m$ we get
  \[
    \sum\limits_{r=0}^{m-1} (-1)^{r+1}{n-m\choose r}{n-r\choose m-r}
    =
    (-1)^m {n-m \choose m} -1
  \]
  and we conclude that
  \[
    S_{\bfa,m}-1=(-1)^m{n-m\choose m}-1.
  \]
\end{proof}

\begin{lemma}
  \label{le:a sums3}
  For $\bfa=(a_1,\dots,a_n)$, assume that $e_\ell(\bfa)=1/\ell!$ whenever $\ell\in[1,n]$.
  Then we have~$S_{\bfa,r}=0$ for $r\in\big[1, \lfloor n/2 \rfloor\big]$.
\end{lemma}
\begin{proof}
  The assumptions imply that $d(m,0)={n\choose m}$ for $m\in[1,n-1]$.
  From $d_\bfa(1,0) = {n \choose 1}$ and~$S_{\bfa,0} =1$, we get
  \[
    S_{\bfa,1} = {n \choose 1} - {n \choose 1}S_{\bfa,0} = 0
  \]
  and we can again work by induction.
  Suppose $S_{\bfa,r}=0$ for $0< r<m \leq \lfloor n/2 \rfloor$ then
  \[
    S_{\bfa,m}
    =
    {n\choose m} - \sum\limits_{r=0}^{m-1} {n-2r \choose m-r} S_{\bfa,r}
    =
    {n\choose m} - {n\choose m} S_{\bfa,0}
    =
    0.
  \]
\end{proof}

\begin{theorem}
  Let $\bfa$ be a point in $(\kk^\times)^n$, and consider the element $x_{n\omega^\circ;c}^\bfa$ of~$\cA_{dp}(c)$.
  \begin{enumerate}
    \item
      \label{hyp:1}
      If $e_\ell(\bfa)=0$ whenever $\ell \in\big[1,\lfloor n/2\rfloor\big]$, then $x_{n\omega^\circ;c}^\bfa$ is equal to the greedy basis element $x_{n\omega^\circ;c}^{gr}$.
    \item
      \label{hyp:2}
      If $e_\ell(\bfa)=1$ whenever $\ell\in[1,n]$, then $x_{n\omega^\circ;c}^\bfa$ is equal to the triangular basis element $x_{n\omega^\circ;c}^{tr}$.
    \item
      \label{hyp:3}
      If $e_\ell(\bfa)=1/\ell!$ whenever $\ell \in[1,n]$, then $x_{n\omega^\circ;c}^\bfa$ is equal to the generic basis element $x_{n\omega^\circ;c}^{ge}$.
  \end{enumerate}
\end{theorem}
\begin{proof}
  Under the hypotheses in \eqref{hyp:1}, the coefficients $d_\bfa(m,0)$ equal $1$ if $m$ is $0$ or $n$ and vanish otherwise; these account for the two terms outside of the summation sign in the expression of $x_{n\omega^\circ;c}^{gr}$.
  Assuming now that $k>0$, Proposition~\ref{prop:imaginary minor} and Lemma~\ref{le:a sums1} combine to give
  \[
    d_\bfa(m,k)=\sum\limits_{r=0}^m (-1)^r \frac{n}{n-r}{m-r\choose k} {n-2r\choose m-r} {n-r\choose r}.
  \]
  But observe that we have the identity
  \begin{align*}
    \frac{n}{n-r}{m-r\choose k} {n-2r\choose m-r} {n-r\choose n-2r}
    & =
    \frac{n(n-r)!}{(n-r)k!(m-r-k)!(n-m-r)!r!}
    \\
    & =
    \frac{n}{k} {m-1\choose k-1} {m-k\choose r} {n-r-1\choose n-m-r}
  \end{align*}
  and so
  \[
    d_\bfa(m,k)
    =
    \frac{n}{k} {m-1\choose k-1} \sum\limits_{r=0}^m (-1)^r {m-k\choose r} {n-r-1\choose n-m-r}.
  \]
  Thus we may apply Lemma~\ref{le:binomial identity} with $a=n-1$, $b=m-k$, and $c=n-m$ to get
  \[
    d_\bfa(m,k)=\frac{n}{k} {m-1\choose k-1} {n-m+k-1\choose n-m}.
  \]
  This is exactly the coefficient in the greedy basis element and we have established \eqref{hyp:1}.

  Similarly, assuming \eqref{hyp:2}, we may combine Proposition~\ref{prop:imaginary minor} and Lemma~\ref{le:a sums2} to get
  \[
    d_\bfa(m,k)
    =
    \sum\limits_{r=0}^m (-1)^r {m-r\choose k} {n-2r\choose m-r} {n-r\choose r}.
  \]
  By the identity
  \[
    {m-r\choose k} {n-2r\choose m-r} {n-r\choose r}
    =
    \frac{(n-r)!}{k!(m-r-k)!(n-m-r)!r!}
    =
    {m\choose k}{m-k\choose r}{n-r\choose n-m-r},
  \]
  we may rewrite $d_\bfa(m,k)$ as
  \[
    d_\bfa(m,k)
    =
    {m\choose k}\sum\limits_{r=0}^m (-1)^r {m-k\choose r}{n-r\choose n-m-r}.
  \]
  Then Lemma~\ref{le:binomial identity} with $a=n$, $b=m-k$, and $c=n-m$ gives
  \[
    d_\bfa(m,k)={m\choose k} {n-m+k\choose n-m}.
  \]
  This is exactly the coefficient in the triangular basis element and we have established \eqref{hyp:2}.

  Finally, combining Proposition~\ref{prop:imaginary minor} and Lemma~\ref{le:a sums3} in case \eqref{hyp:3} we get
  \[
    d_\bfa(m,k)={m\choose k} {n\choose m}.
  \]
  This is exactly the coefficient in the generic basis element and \eqref{hyp:3} follows.
\end{proof}

\begin{remark}
  For any weight $\lambda$ of $\wh{G}$, a universal representation $V(\lambda)$ for which $\lambda$ is extremal was introduced in \cite{Kas94}.
  For the weight $\lambda = n \omega^\circ$, this can be identified with the $n$-th symmetric power of $V[u^{\pm 1}]$, where again $V$ is the vector representation of $SL_2$ \cite{CP01}.
  Let $\Delta_{V(n \omega^\circ),n \omega^\circ}$ be the minor defined by $v_{n\omega^\circ} = v \otimes \cdots \otimes v$ and suppose $\pi_{n\omega^\circ}$ annihilates non-constant elementary symmetric elements of $\widetilde{LSL}_2$-weight ${n\omega^\circ}$.
  Then one can show that $x_{n\omega^\circ;c}^{ge}$ also coincides with the restriction of~$\Delta_{V(n \omega^\circ),n \omega^\circ}$ to $\widetilde{LSL}_2^c$.
\end{remark}


\newcommand{\etalchar}[1]{$^{#1}$}
\providecommand{\bysame}{\leavevmode\hbox to3em{\hrulefill}\thinspace}
\providecommand{\MR}{\relax\ifhmode\unskip\space\fi MR }
\providecommand{\MRhref}[2]{%
  \href{http://www.ams.org/mathscinet-getitem?mr=#1}{#2}
}
\providecommand{\href}[2]{#2}


\end{document}